\documentclass[11.5pt]{amsart}

\usepackage[utf8]{inputenc}
\usepackage{csquotes}
\usepackage{array,amsbsy,amscd,amsfonts,color,amssymb,amstext,amsmath,latexsym,amsthm,amscd,multicol,graphicx, tikz,tikz-cd, physics}
\usepackage[english]{babel}

\usepackage[mathscr]{eucal}

\usepackage{amssymb,mathscinet}
\usepackage[all]{xy}

\numberwithin{equation}{section}

\usepackage{amsfonts,amssymb,amscd,amsmath,enumerate,verbatim,calc,eucal,enumerate}

\newcommand{\B}{\mathcal{B}}

\renewcommand{\H}{\mathcal{H}}

\theoremstyle{plain}
\newtheorem{theorem}{Theorem}[section]
\newtheorem{corollary}[theorem]{Corollary}
\newtheorem{lemma}[theorem]{Lemma}

\newtheorem{definition}[theorem]{Definition}

\newtheorem{remark}[theorem]{Remark}
\newtheorem{example}[theorem]{Example}

\begin{document}
	
	\title[]
	{Block quantum dynamical semigroups of completely positive definite kernels}

	\author[Dey]{Santanu Dey \textsuperscript{*}}
	\address{Department of Mathematics, Indian Institute of Technology Bombay, Mumbai-400076, India}
	\email{santanudey@iitb.ac.in}
	\author[Saini]{Dimple Saini}
	\address{Center for Mathematical and Financial Computing, Department of Mathematics, The LNM Institute of Information Technology, Rupa ki Nangal, Post-Sumel, Via-Jamdoli
		Jaipur-302031,
		(Rajasthan) INDIA}
	\email{18pmt006@lnmiit.ac.in,  dimple92.saini@gmail.com}
	\author[Trivedi]{Harsh Trivedi}
	\address{Center for Mathematical and Financial Computing, Department of Mathematics, The LNM Institute of Information Technology, Rupa ki Nangal, Post-Sumel, Via-Jamdoli
		Jaipur-302031,
		(Rajasthan) INDIA}
	\email{harsh.trivedi@lnmiit.ac.in, trivediharsh26@gmail.com}
	\thanks{*corresponding author}

	%\tableofcontents

	\begin{abstract}
		Kolmogorov decomposition for a given completely positive definite kernel is a generalization of Paschke's GNS construction for the completely positive map. Using Kolmogorov decomposition, to every quantum dynamical semigroup (QDS) for completely positive definite kernels over a set $S$ on given $C^*$-algebra $\mathcal{A},$ we shall assign an inclusion system $F = (F_s)_{s\ge 0}$ of Hilbert bimodules over $\mathcal{A}$ with a generating unit $\xi^{\sigma}=(\xi^{\sigma}_s)_{s\ge 0}.$ Consider a von Neumann algebra $\mathcal{B}$, and let $\mathfrak{T}=(\mathfrak{T}_s)_{s\ge 0}$
		be a QDS over a set $S$ on the algebra $M_2(\mathcal{B})$ with $\mathfrak{T}_s=\begin{pmatrix}
			\mathfrak{K}_{s,1} & \mathfrak{L}_s\\\mathfrak{L}_s^*& \mathfrak{K}_{s,2}
		\end{pmatrix}$ which acts block-wise. Further, suppose that $(F^i_s )_{s\ge 0}$ is the inclusion system affiliated to the diagonal QDS $(\mathfrak{K}_{s,i})_{s\ge 0}$ along with the generating unit $(\xi^{\sigma}_{s,i} )_{s\ge 0},$ $\sigma\in S,i\in \{1,2\}$, then we prove that there exists a unique contractive (weak) morphism $V = (V_s)_{s\ge 0}:F^2_s \to F^1_s$ 
		such that $\mathfrak{L}_s^{\sigma,\sigma'}(b)=\langle \xi_{s,1}^{\sigma},V_s b\xi_{s,2}^{\sigma'}\rangle$ for every $\sigma',\sigma\in S$ and $b\in \mathcal{B}.$ We also study the semigroup version of a factorization theorem for $\mathfrak{K}$-families.
	\end{abstract}


	\maketitle
	
	\section{Introduction}
	In Matrix Analysis, it is clear that a given block matrix $\begin{pmatrix}
		X & Y\\ Y^*& Z
	\end{pmatrix}$ defined on $\mathcal{H}_1\oplus \mathcal{H}_2$ is a positive block matrix if and only if $X$ and $Z$ are positive operators as well as there is an operator $D:\mathcal{H}_2\to \mathcal{H}_1$ which is a contraction with $Y=X^{1/2}DZ^{1/2}.$

A fundamental construction due to Gelfand-Naimark-Segal (GNS), see \cite{Pau02}, says that any given positive linear functional $w$ on a unital $C^*$-algebra $\mathcal{B}$ is characterized by a Hilbert space $\mathcal{H},$ a generating (cyclic) unit vector $\Omega$ in $\mathcal{H},$ and a representation $\pi:\mathcal{B}\to B(\mathcal{H})$ such that $w(b)=\langle \Omega,\pi(b)\Omega\rangle$ for all $b\in \mathcal{B}.$ Let $\mathcal{A}$ be a  $C^*$-algebra, then a linear mapping $\phi:\mathcal{A}\to\mathcal{B}$ is called {\it completely positive} if $\sum_{i,j=1}^n b^*_j\phi(a_j^*a_i)b_i\geq 0$ for every $n\in\mathbb N,$ where $b_j,b_i\in\mathcal{B}$ and $a_j,a_i\in\mathcal{A}.$ Completely positive maps, are generalization of positive linear functionals, and their decompositions have applications to Quantum Statistical Mechanics, Operator Algebras,  Quantum Information Theory, and other related research areas. Stinespring \cite{Stin55} gave a characterization of all $B(\mathcal{H})$-valued completely positive maps and this result was further extended for $\mathcal{B}$-valued completely positive maps by Paschke \cite{Pas73} where $\mathcal{B}$ is a von Neumann algebra the proof is similar to that of the GNS construction, but instead of a Hilbert space here we get a Hilbert $\mathcal{A}$-$\mathcal{B}$-bimodule. Considering the block completely positive maps (by taking $\mathcal{B}=B(\mathcal{H})$) Paulsen \cite{Pau84} proved that every bounded operator on a Hilbert space is completely polynomially bounded if and only if it is similar to a contraction and hence gave a partial answer to a question mentioned by Halmos \cite{H55} which inquires if each polynomially bounded operator on a Hilbert space is similar to a contraction. 

Let $\phi:\mathcal{A}\to B(\mathcal{H})$ be a completely positive map and $(\mathcal{K},\pi,V)$ be its minimal Stinespring representation where $V:\mathcal{H}\to\mathcal{K}$ is an isometry and $\pi:\mathcal{A}\to B(\mathcal{K})$ is a representation, and let $\begin{pmatrix}
		\phi & \psi\\ \psi^*& \phi
	\end{pmatrix}$ from $M_2(\mathcal{A})$ into $M_2(B(\mathcal{H}))$ be a block completely positive map, then Paulsen and Suen in \cite{PC85} proved that
 there is a contractive map $T\in \pi(\mathcal{A})'$ with $V^*\pi(b)TV=\psi(b)$ for $b\in \mathcal{A}.$ This result may be considered as an analogue of the above mentioned result on positivity of block matrices $\begin{pmatrix}
 	X & Y\\ Y^*& Z
 \end{pmatrix}$. Furuta \cite{F94} explored the partial matrices completion problem of block completely positive functions. 
	 Bhat and Kumar \cite{BK20} extended the Paulsen and Suen result based on the Paschke's GNS construction in the following way:
	\begin{theorem}$($Bhat-Kumar$)$\label{DS4}
		Let $\mathcal{A}$ be a von Neumann algebra on a Hilbert space $\mathcal{H}$ and $\mathcal{B}$ be a $C^*$-algebra and let $\phi_i:\mathcal{B}\to \mathcal{A}$ be a completely positive function with GNS representation $(E_i,x_i),$ where $i\in \{1,2\}.$ If $\begin{pmatrix}
			\phi_1 & \psi\\ \psi^*& \phi_2
		\end{pmatrix}$ from $M_2(\mathcal{B})$ into $M_2(\mathcal{A})$ is a block completely positive map for some completely bounded (or CB) map $\psi:\mathcal{B}\to \mathcal{A},$ then there exists a bilinear adjointable contractive map $V:E_2\to E_1$ with $\psi(b)=\langle x_1,Vbx_2\rangle$ for every $b\in \mathcal{B}.$ 
	\end{theorem}

 Heo \cite{He99} introduced completely multi-positive linear maps. In \cite{Ske11} Skeide mentioned that the notion of completely positive definite kernel is a generalization of completely multi-positive linear map.	This article is about the structure of a semigroup of block completely positive definite kernels. 
	\begin{definition}
		Let $\mathcal{B}$ and $\mathcal{A}$ be two $C^*$-algebras. A function $\mathfrak{K}:S\times S \to B(\mathcal{A},\mathcal{B})$ is said to be {\rm completely positive definite kernel,} or briefly {\rm CPD-kernel} over a set $S$ if
		\[
		\sum_{j,i} b^*_j \mathfrak{K}^{\sigma_j, \sigma_i} (a^*_j a_i) b_i \geq 0~\mbox{for finite $\sigma_i\in S$, $b_i\in \mathcal{B}$, $a_i\in \mathcal{A}$. }~
		\]
	\end{definition} Suppose that $\mathcal{B}$ is a $C^*$-algebra. A (right) $\mathcal{B}$-module vector space $E$ is said to be {\it Hilbert $\mathcal{B}$-module} if it is a $\mathcal{B}$-valued inner product, which is complete with respect to the associated norm (see \cite{L95, Pas73, Sk00}). If the closed linear span of $\{\langle z,w \rangle : z,w\in E\}$ equals $\mathcal{B},$ then we say that $E$ is {\it full}. The following Kolmogorov decomposition from \cite{BBLS04} is a generalization of Paschke's GNS construction and it resembles the construction of reproducing kernel Hilbert space:

\begin{theorem}[Barreto-Bhat-Liebscher-Skeide] Let $\mathfrak{K}:S\times S\to B(\mathcal{A},\mathcal{B})$ be a CPD-kernel over a given set $S,$ then there is a {Hilbert $\mathcal{A}$-$\mathcal{B}$-module} $F$ and a function $\mathfrak{i}:S\to F$ such that $\mathfrak{K}^{\sigma,\sigma'}(a)=\langle \mathfrak{i}(\sigma), a\mathfrak{i}(\sigma')\rangle$ for all $a\in \mathcal{A}$ and $\sigma',\sigma\in S$ with the minimal condition is  $F=\overline{span}(\mathcal{A}\mathfrak{i}(\sigma)\mathcal{B}).$ 
	\end{theorem}
	We say that $(F,\mathfrak{i})$ is the {\it Kolmogorov-representation} of $\mathfrak{K}$ over $S.$ Kolmogorov decomposition is an active area of research particularly in the study of  type $I$ product systems of Hilbert modules, Noncommutative reproducing kernel Hilbert spaces as well as correspondences, transfer function realization for the noncommutative Schur-Agler-Class, and the study of non commutative Szego kernel, for more details see \cite{BFH08,BMV16,BBLS04,MS08,S20}.
	
	A continuous time version of the Stinespring’s theorem for completely positive functions was derived by Parthasarathy in \cite{KRP90}. A continuous time version of Paschke's GNS construction was obtained in \cite{BS2000} and of Kolmogorov construction for CPD kernels was obtained in \cite{BBLS04}. Bhat and Kumar proved continuous time version of Theorem \ref{DS4} by considering semigroup of block CP maps. In Section 2, we study a result about the structure of block completely positive definite kernels based on Theorem \ref{DS4} and in Section 3, we study continuous time version and relevant results. In Section 4, we present a lifting theorem based on \cite{BK20}.
	
	Let us recall the following notion of $\phi$-maps:
 \begin{definition} Consider Hilbert $C^*$-modules ${E}$ and ${F}$ over $C^*$-algebras $\mathcal{B}$ and $\mathcal{C},$ respectively. Let $\phi:\mathcal{B}\to \mathcal{C}$ be a linear function. A map $T:E\to F$ is said to be {\rm $\phi$-map} if $\phi(\langle y,x\rangle)=\langle T(y),T(x)\rangle$ for each $y,x\in E.$ \end{definition}
Bhat, Ramesh and Sumesh in \cite{BRS12} extended the Stinespring theorem for $\phi$-maps. In \cite{Sk12}, using Paschke's GNS construction, Skeide proved a factorization theorem for $\phi$-maps which is a generalization of the Bhat, Ramesh and Sumesh.
	We recall the definition of $\mathfrak{K}$-family from  \cite{DT17} which is a generalization of $\phi$-maps:
	\begin{definition}Let ${E}$ and ${F}$ be Hilbert $C^*$-modules over $\mathcal{B}$ and $\mathcal{C},$ respectively. Suppose $S$ is a set and $ \mathfrak{K}:S\times S\to B(\mathcal{B},\mathcal{C})$ is a kernel. For each $\sigma\in S,$ consider a map $\mathcal{K}^{\sigma}:E\to F,$ we say that the family $\{\mathcal{K}^{\sigma}\}_{\sigma\in S}$ is a {\rm $\mathfrak{K}$-family} if
	$\mathfrak{K}^{\sigma,\sigma'}(\langle y, x\rangle)=\langle \mathcal{K}^{\sigma} (y),\mathcal{K}^{{\sigma}'}(x)\rangle~\mbox{for all}~ \sigma',\sigma\in S$ and $y,x\in E.$
	\end{definition}
	Dey and Trivedi \cite{DT17} generalized Skeide's factorization theorem for $\mathfrak{K}$-families as follows:
	\begin{theorem}\label{1234}
		Suppose that $\mathcal{B}$ and $\mathcal{C}$ are two $C^*$-algebras where $\mathcal{B}$ is unital. Let $S$ be a set, and ${F}$ and ${E}$ be Hilbert $C^*$-modules over $\mathcal{C}$ and $\mathcal{B},$ respectively. Suppose that $\mathcal{K}^{\sigma}:E\to F$ is a map for every $\sigma\in S.$ Then the following conditions are equivalent:
		\begin{enumerate}
			\item  There is a CPD-kernel $ \mathfrak{K}:S\times S\to B(\mathcal{B},\mathcal{C})$ with $\{\mathcal{K}^{\sigma}\}_{\sigma\in S}$ is a $\mathfrak{K}$-family.% The family $\{\mathcal{K}^{\sigma}\}_{\sigma\in S}$ is a $\mathfrak{K}$-family, where the kernel $ \mathfrak{K}:S\times S\to B(\mathcal{B},\mathcal{C})$ is CPD.
			\item There exist an isometry $v:E \odot \mathcal{F}\to F$ where $\mathcal{F}$ is a $\mathcal{B}$-$\mathcal{C}$-correspondence, and a function $\mathfrak i:S\to \mathcal{F}$  with $$\mathcal{K}^{\sigma}(y)=v(y\odot \mathfrak i (\sigma)) \quad \quad \mbox{for} \quad  \sigma\in S,y\in E.$$
		\end{enumerate}
	\end{theorem}
Indeed, based on the work of Skeide and Sumesh \cite{SSu14}, it was proved in \cite{DT17} that a $\mathfrak{K}$-family has several alternate characterization as mentioned in next theorem.
\begin{theorem}\label{DS5}
	Suppose that ${E}$ and ${F}$ are Hilbert $C^*$-modules over $C^*$-algebras $\mathcal{B}$ and $\mathcal{C},$ respectively, where $E$ is full. Suppose that $\mathcal{K}^{\sigma}:E\to F$ is a linear map for every $\sigma\in S$ and set $F_{\mathcal{K}}:=\overline{span}\{\mathcal{K}^\sigma(y)c:~\sigma\in S,~c\in \mathcal{C},y\in E\}$. Then the following conditions are equivalent:
	\begin{enumerate}
		\item [(a)] There is a unique CPD-kernel $\mathfrak{K}:S\times S\to B(\mathcal{B},\mathcal{C})$  so that $\{\mathcal{K}^{\sigma}\}_{\sigma\in S}$ is a $\mathfrak{K}$-family.
		\item [(b)] $\{\mathcal{K}^{\sigma}\}_{\sigma\in S}$ extends to a block wise CPD-kernel $\begin{pmatrix}
			{\mathfrak{K}^{\sigma,\sigma'}} &  {\mathcal{K}^{\sigma^*}} \\
			\mathcal{K}^{\sigma'} & \vartheta
		\end{pmatrix}:\begin{pmatrix}
			\mathcal{B} &  E^* \\
			E & B^a(E)
		\end{pmatrix}\to \begin{pmatrix}
			\mathcal{C} &  F_{\mathcal{K}}^* \\
			F_{\mathcal{K}} & B^a(F_{\mathcal{K}})
		\end{pmatrix},$ where $\vartheta$ is a $*$-homomorphism.
		\item [(c)] For $s_1,\ldots,s_n\in S$ the function from $E_n$ to $F_n$ given by
		$${\bf y}\mapsto({\mathcal{K}^{s_1}}(y_1),{\mathcal{K}^{s_2}}(y_2),\ldots,{\mathcal{K}^{s_n}}(y_n))~\mbox{for every}~{\bf y}=(y_1,y_2,\ldots,y_n)\in E_n$$
		is a completely bounded map and $F_{\mathcal{K}}$ can be made into a
		$B^a ({E})$-$\mathcal{C}$-correspondence so that $\mathcal{K}^{\sigma}$ is left $B^a({E})$-linear function.
		
		\item [(d)] For $s_1,\ldots,s_n\in S$ the function from $E_n$ to $F_n$ given by
		$${\bf y}\mapsto({\mathcal{K}^{s_1}}(y_1),{\mathcal{K}^{s_2}}(y_2),\ldots,{\mathcal{K}^{s_n}}(y_n))~\mbox{for all}~{\bf y}=(y_1,y_2\ldots,y_n)\in E_n$$
		is completely bounded  and $\{\mathcal{K}^{\sigma}\}_{\sigma\in S}$
		is such that
		$$\langle \mathcal{K}^\sigma{( y)},\mathcal{K}^{\sigma'}({ w}\langle { w}',{ y}'\rangle)\rangle=\langle \mathcal{K}^{\sigma}({ w}'\langle {w},{ y}\rangle),\mathcal{K}^{\sigma'}({y}')\rangle~\mbox{for all ${ w}',{ y}',w,y\in{E},\sigma',\sigma\in S$.}$$
		
	\end{enumerate}
	
\end{theorem}
 In Section 5, we prove the semigroup version of the factorization theorem for $\mathfrak{K}$-families.

	\subsection{Preliminaries and Notations}
	 Throughout this paper, we will use the following notations: $B(\mathcal{H},\mathcal{K})$ for the algebra of all bounded linear operators between Hilbert spaces $\mathcal{H}$ and $\mathcal{K},$ $B^a({E,F})$ for the algebra of all adjointable operators from Hilbert $\mathcal{B}$-module $E$ to Hilbert $\mathcal{B}$-module $F$ and $B^a({E})$ for $B^a({E,E}).$

	Suppose that ${\mathcal{B}}$ is a von Neumann algebra (or simply say, vNa) on $\mathcal{H}.$ A Hilbert ${\mathcal{B}}$-module $F$ is called {\it von Neumann ${\mathcal{B}}$-module} if it is strongly closed in $B({\mathcal{H},F\odot \mathcal{H}}).$ Moreover, suppose that ${\mathcal{A}}$ is a vNa, a von Neumann $\mathcal{B}$-module $F$ is called {\it von Neumann $\mathcal{A}$-$\mathcal{B}$-module} if $F$ is a Hilbert $\mathcal{A}$-$\mathcal{B}$-module with the normal Stinespring representation $\rho: \mathcal{A} \to B(F\odot \mathcal{H}).$ If $F$ is a von Neumann $\B$-module, then $B^a(F)$ is a von Neumann subalgebra of $B(F\odot \H).$ von Neumann modules are self-dual and hence any bounded right linear map between von Neumann modules is adjointable (see \cite{Sk05,BK20}). %For more details see \cite{Pas73, Sk00, Sk05}.

	Suppose that $\mathfrak{K}:S\times S\to B(\mathcal{A},\mathcal{B})$ is a CPD-kernel over a set $S,$ where $\mathcal{A}$ and $\mathcal{B}$ are unital $C^*$-algebras. Let $(E,\mathfrak{i})$ be the minimal Kolmogorov-representation for $\mathfrak{K}$. Suppose that $\mathcal{B}$ is a concrete $C^*$-algebra of operators acting nondegenerately on $\mathcal{H}$. Let ${K}=E\odot \mathcal{H}$ and denote by $\rho:a \mapsto a \odot I_{\mathcal{H}}$ be the Stinespring representation of $\mathcal{A}$ on ${K}.$ For $\sigma \in S,$ defined a bounded map $L_{\mathfrak{i}(\sigma)}:\mathcal{H}\to K$ by $L_{\mathfrak{i}(\sigma)}(h)=\mathfrak{i}(\sigma) \odot h$ for $h\in \mathcal{H}$ with $L_{\mathfrak{i}(\sigma)}^*:\mathfrak{i}(\sigma')\odot h \mapsto \langle \mathfrak{i}(\sigma),\mathfrak{i}(\sigma')\rangle h.$ Define a function $\xi:E\to B(\mathcal{H},K)$ by $\xi(\mathfrak{i}(\sigma))=L_{\mathfrak{i}(\sigma)}$, then $L_{\mathfrak{i}(\sigma)}^*L_{\mathfrak{i}(\sigma')}=\langle \mathfrak{i}(\sigma),\mathfrak{i}(\sigma')\rangle$ in $B(\mathcal{H})$ and $L_{a\mathfrak{i}(\sigma)b}=\rho(a)L_{\mathfrak{i}(\sigma)} b.$ We can recognize $E$ as a concrete subset of $B(\mathcal{H},K).$ For $\sigma',\sigma\in S$ we have $$\mathfrak{K}^{\sigma,\sigma'}(a)=\langle \mathfrak{i}(\sigma),a\mathfrak{i}(\sigma')\rangle=L_{\mathfrak{i}(\sigma)}^*L_{a\mathfrak{i}(\sigma')}=L_{\mathfrak{i}(\sigma)}^*\rho(a)L_{\mathfrak{i}(\sigma')}\quad \mbox{for all} \quad a\in \mathcal{A}.$$ This implies that $\mathfrak{K}^{\sigma,\sigma'}$ is unital if and only if $L_{\mathfrak{i}(\sigma)}$ is an isometry in $B(\mathcal{H},K).$ Note that $\overline{span}\{\rho(a)L_{\mathfrak{i}(\sigma)} h: a\in \mathcal{A},h\in\mathcal{H},\sigma\in S\}=\overline{span}\{a{\mathfrak{i}(\sigma)}\odot h: a\in \mathcal{A},h\in\mathcal{H},\sigma\in S\}=E\odot \mathcal{H}=K.$ We say that representation $(K,\rho,L_{\mathfrak{i}(\sigma)})$ is the (minimal) Kolmogorov Stinespring representation for $\mathfrak{K}$ over $S.$ On the other hand, suppose that $(K, \rho, L_{\mathfrak{i}(\sigma)})$ is the (minimal) Kolmogorov Stinespring representation for $\mathfrak{K}$ over $S.$ Consider $B(\mathcal{H},K)$ as
	a Hilbert $\mathcal{A}$-$B(\mathcal{H})$-module with the left action of $\mathcal{A}$ is defined by $\rho$. Suppose $E=\overline{span}\mathcal{A} L_{\mathfrak{i}(\sigma)}B(\mathcal{H}) \subseteq B(\mathcal{H},K),$ then $(E,L_{\mathfrak{i}(\sigma)})$ is the (minimal) Kolmogorov-representation for $\mathfrak{K}$ over $S.$  For $k\in \mathcal{K},g\in \mathcal{H},$ a bounded function $|k \rangle\langle g|: \mathcal{K}\to \mathcal{H}$ given by $|k \rangle\langle g|(h)=\langle k,h\rangle g$ for every $h\in \mathcal{K}.$ Two Hilbert $C^*$-modules $F$ and $E$ are said to be {\it isomorphic} if there exists a bilinear unitary between them, and denoted by $F \simeq E.$

	\section{Block completely positive definite kernels}\label{sec1}
	
	In this section, we will discuss the structure of block completely positive definite kernels. Suppose that $\mathcal{B}$ is a unital $C^*$-algebra with $p\in\mathcal{B}$ is a projection and $q=\mathbf{1}-p,$ then the block decomposition for $b\in \mathcal{B}$ as follows:
	\begin{equation}
		b = \begin{pmatrix}
			p b p & pb q \\ q b p & q b q
		\end{pmatrix}\in \begin{pmatrix}
			p\mathcal{B} p&  p\mathcal{B} q \\ q \mathcal{B} p & q \mathcal{B} q
		\end{pmatrix}.
	\end{equation}
	
	\begin{definition}
		Suppose that $\mathcal{A}$ and $\mathcal{B}$ are two unital $C^*$-algebras and $p_1\in\mathcal{A}$ and $p_2\in\mathcal{B}$ are two projections. A function $ \mathfrak{K}:S\times S\to  B(\mathcal{A},\mathcal{B})$ is called \emph{block kernel} over a set $S$ if $\mathfrak{K}$ respects the above block decomposition over $S,$ that is, for every $a\in\mathcal{A}$ and $\sigma',\sigma\in S,$ we have
		\begin{equation}
			\mathfrak{K}^{\sigma,\sigma'}(a) = \begin{pmatrix}
				\mathfrak{K}^{\sigma,\sigma'}(p_1 a p_1) & \mathfrak{K}^{\sigma,\sigma'}(p_1a q_1) \\ \mathfrak{K}^{\sigma,\sigma'}(q_1 a p_1) & \mathfrak{K}^{\sigma,\sigma'}(q_1 a q_1)
			\end{pmatrix}\in \begin{pmatrix}
				p_2\mathcal{B} p_2&  p_2\mathcal{B} q_2  \\q_2  \mathcal{B} p_2 &q_2  \mathcal{B} q_2
			\end{pmatrix}.
		\end{equation}
	\end{definition}
	
	Suppose that $\mathfrak{K}:S\times S\to  B(\mathcal{A},\mathcal{B})$ is a block kernel. Then there exist
	$\mathfrak{K}^{\sigma,\sigma'}_{11}:p_1\mathcal{A} p_1\to p_2\mathcal{B} p_2,$ $\mathfrak{K}^{\sigma,\sigma'}_{12}:p_1 \mathcal{A} q_1 \to p_2\mathcal{B} q_2,\mathfrak{K}^{\sigma,\sigma'}_{21}: q_1 \mathcal{A} p_1\to q_2 \mathcal{B} p_2,$ and $\mathfrak{K}^{\sigma,\sigma'}_{22}:q_1 \mathcal{A} q_1\to q_2 \mathcal{B} q_2$ such that
	\[\mathfrak{K}^{\sigma,\sigma'}=\begin{pmatrix}
		\mathfrak{K}^{\sigma,\sigma'}_{11} &\mathfrak{K}^{\sigma,\sigma'}_{12} \\\mathfrak{K}^{\sigma,\sigma'}_{21} &\mathfrak{K}^{\sigma,\sigma'}_{22}
	\end{pmatrix} \quad \quad for \quad \sigma',\sigma\in S .\]

	\begin{lemma}\label{lem-single}
		Suppose that $\mathcal{A}$ and $\mathcal{B}$ are two unital $C^*$-algebras and we denote $I_2:=\{1,2\}.$ Let $ \mathfrak{K}_i:S\times S\to  B(\mathcal{A},\mathcal{B})$ be a CPD-kernel over $S$ with Kolmogorov-representation $(F_i,\mathfrak{j}_i)$ for $i\in I_2,$ and let $V: F_2\to F_1$ be a bilinear adjointable contraction and a kernel $\mathfrak{L}:S\times S\to B(\mathcal{A},\mathcal{B})$ is defined by $\mathfrak{L}^{\sigma,\sigma'}(a)=\langle \mathfrak{j}_1(\sigma), Va\mathfrak{j}_2(\sigma')\rangle.$ Then the block kernel $\mathfrak{K}=\begin{pmatrix}
			\mathfrak{K}_1 & \mathfrak{L}\\\mathfrak{L}^*& \mathfrak{K}_2
		\end{pmatrix}: S\times S\to  B(M_2(\mathcal{A}), M_2(\mathcal{B}))$ is CPD over $S$.
	\end{lemma}
	
	\begin{proof}
		Let $\mathfrak{k} (\sigma)=V\mathfrak{j}_2(\sigma)\in F_1$ and $\sigma',\sigma\in S$ we have
		\begin{equation*}
			\mathfrak{K}^{\sigma,\sigma'} \begin{pmatrix}
				a_{11} &a_{12}\\a_{21}&a_{22}
			\end{pmatrix}
			= \begin{pmatrix}
				\langle \mathfrak{j}_1(\sigma),a_{11} \mathfrak{j}_1(\sigma') \rangle&\langle \mathfrak{j}_1(\sigma),a_{12} \mathfrak{k}(\sigma') \rangle\\\langle \mathfrak{k}(\sigma'), a_{21}\mathfrak{j}_1(\sigma)\rangle&\langle \mathfrak{k}(\sigma), a_{22}\mathfrak{k}(\sigma')\rangle
			\end{pmatrix} +\begin{pmatrix}
				0&0\\0&\langle \mathfrak{j}_2(\sigma), a_{22}(I_{F_2}-V^*V)\mathfrak{j}_2(\sigma')\rangle
			\end{pmatrix}
		\end{equation*} for all $\begin{pmatrix}
			a_{11} &a_{12}\\a_{21}&a_{22}
		\end{pmatrix}\in M_2(\mathcal{A}).$
		It is easy to verify that $\begin{pmatrix}
			a_{11} &a_{12}\\a_{21}&a_{22}
		\end{pmatrix}\mapsto\begin{pmatrix}
			\langle \mathfrak{j}_1(\sigma),a_{11} \mathfrak{j}_1(\sigma') \rangle&\langle \mathfrak{j}_1(\sigma),a_{12} \mathfrak{k}(\sigma') \rangle\\\langle \mathfrak{k}(\sigma), a_{21}\mathfrak{j}_1(\sigma')\rangle&\langle \mathfrak{k}(\sigma), a_{22}\mathfrak{k}(\sigma')\rangle
		\end{pmatrix} $ is a CPD-kernel over $S$.
		Since $(I_{F_2}-V^*V)$ is positive and bilinear,  $$\begin{pmatrix}
		a_{11} &a_{12}\\a_{21}&a_{22}
		\end{pmatrix}\mapsto\begin{pmatrix}
			0&0\\0&\langle \mathfrak{j}_2(\sigma), a_{22}(I_{F_2}-V^*V)\mathfrak{j}_2(\sigma')\rangle
		\end{pmatrix}$$ is also a CPD-kernel over $S$. Therefore $\mathfrak{K}$ is a CPD-kernel over $S$.
	\end{proof}
	
	Suppose that $E$ is a Hilbert $M_2(\mathcal{B})$-module. Define a $\mathcal{B}$-valued semi-inner product  $\langle \cdot,\cdot\rangle_\Sigma$ and a right $\mathcal{B}$-module action on $E$ by
	\begin{equation*}\label{eq-right-act}
		\langle u,v\rangle_\Sigma:=\sum\limits_{s,r=1}^{2}\langle u,v\rangle_{r,s}  \text{ and }	ub:=u\begin{pmatrix}
			b  & 0\\0 & b
		\end{pmatrix} \quad\text{for }  b \in \mathcal{B}, u,v\in E,
	\end{equation*}
and the notation $\langle u,v\rangle_{r,s}$ indicates the $(r,s)$\textsuperscript{th} entry of $\langle u,v\rangle\in M_2(\mathcal{B}).$	

We denote $E^{(\mathcal{B})}$ for quotient space $E/M$ with $M=\{u:\langle  u,u\rangle_\Sigma=0\},$ and $[u]$ (or $[u]_E$) for coset $u+M$ of $u\in E.$ Note that $E^{(\mathcal{B})}$ is a pre-Hilbert $\mathcal{B}$-module with the inner product and right $\mathcal{B}$-module action defined by
	\begin{equation}
		\langle [u],[v]\rangle=\langle u,v\rangle_\Sigma=\sum_{r,s=1}^2 \langle u,v\rangle_{r,s}  \text{ and }  	[u]b=[u\begin{pmatrix}
			b  & 0\\0 & b
		\end{pmatrix}] \quad\text{for }\quad b\in\mathcal{B},u,v\in E.
	\end{equation}
	Further, assume that $E$ is a Hilbert $M_2(\mathcal{B})$-module (resp. von Neumann $M_2(\mathcal{B})$-module), then $E^{(\mathcal{B})}$ is a Hilbert $\mathcal{B}$-module (resp. von Neumann $\mathcal{B}$-module). 
	
	Suppose that $E$ is a Hilbert $M_2(\mathcal{B})$-module with a nondegenerate left module action of $\mathcal{A},$ then the canonical left module action of $\mathcal{A}$ on $E^{(\mathcal{B})}$ defined by
	\begin{equation}\label{new-left-act}
		b [u]:=[bu]\quad\text{for } u\in E,b\in\mathcal{A}.
	\end{equation}
	Moreover, if $F$ is a $M_2(\mathcal{B})$-submodule of $M_2(\mathcal{B})$-module $E,$ then $F^{(\mathcal{B})}\simeq \{[u]_E:u\in F\}\subseteq E^{(\mathcal{B})}.$ 

	\begin{remark}\label{obs-new-module}
		Let $E$ be a (von Neumann) Hilbert $\mathcal{A}$-$M_2(\mathcal{B})$-module. Then $E^{(\mathcal{B})}$ is a (von Neumann) Hilbert $\mathcal{A}$-$\mathcal{B}$-module, where the left module action given in Equation \eqref{new-left-act}. Further, assume that $E$ is a (von Neumann) Hilbert $M_2(\mathcal{A})$-$M_2(\mathcal{B})$-module, then $E$ is consider as $\mathcal{A}$-$M_2(\mathcal{B})$-module, where the left module action of $\mathcal{A}$ is defined by
			\begin{equation*}
				bu:=\begin{pmatrix}
					b &0\\0&b
				\end{pmatrix}u\quad \text{for} \quad u\in E, b\in \mathcal{A},
			\end{equation*}
		and hence $E^{(\mathcal{B})}$ is a (von Neumann) Hilbert $\mathcal{A}$-$\mathcal{B}$-module.
	
		\end{remark}

	Now we are ready to state the main result of this section.
	\begin{theorem}\label{main-single}
		Suppose that $\mathcal{A}$ is a unital $C^*$-algebra and $\mathcal{B}$ is a vNa on $\mathcal{H}.$ Suppose $\mathfrak{K}_i:S\times S\to B(\mathcal{A},\mathcal{B})$ is a CPD-kernel with Kolmogorov-representation $(E_i,\mathfrak{j}_i)$ for $i\in I_2.$ Let $\mathfrak{K}=\begin{pmatrix}
			\mathfrak{K}_1 & \mathfrak{L}\\\mathfrak{L}^*& \mathfrak{K}_2
		\end{pmatrix}:S\times S\to  B(M_2(\mathcal{A}), M_2(\mathcal{B}))$ be a block CPD-kernel for some CB-kernel $\mathfrak{L}:S\times S\to  B(\mathcal{A},\mathcal{B}),$ then there exists a bilinear adjointable contraction $V:E_2\to E_1 $ with $\mathfrak{L}^{\sigma,\sigma'}(a)=\langle \mathfrak{j}_1(\sigma), Va \mathfrak{j}_2(\sigma')\rangle$ for every $\sigma',\sigma\in S$ and $a\in \mathcal{A}.$
	\end{theorem}
	\begin{proof}
		Suppose that $(F,\mathfrak{k})$ is a (minimal) Kolmogorov-representation for $\mathfrak{K}$ over $S,$ then $F$ is a von Neumann $M_2(\mathcal{B})$-module as well as Hilbert $M_2(\mathcal{A})$-$M_2(\mathcal{B})$-module. For $i,j\in I_2,$ set $\mathbb{F}_{ij}:=\mathbf{1} \otimes {\bf e}_{ij}$  in $ \mathcal{B}\otimes M_2,$ or $\mathcal{A}\otimes M_2$,
		where $\{{\bf e}_{ij}\}$'s denote the units of $M_2.$ Clearly $\mathbb{F}_{ii}$'s are projections. Let $\hat{F}_i:=\mathbb{F}_{ii} F\subset F,i\in I_2,$ then $\hat{F}_i$'s are $M_2(\mathcal{B})$-submodules of $F$ with SOT closed and $F=\hat{F}_1\oplus \hat{F}_2.$
		
		Let  $\mathfrak{k}_i(\sigma):=\mathbb{F}_{ii}\mathfrak{k}(\sigma)\mathbb{F}_{ii}\in \hat{F}_i, i\in I_2$ and $\sigma\in S,$ then $\langle \mathfrak{k}_1(\sigma),\mathfrak{k}_2(\sigma')\rangle=0$ for $\sigma',\sigma\in S.$ For $i,j\in I_2,\sigma\in S$ and $i\ne j$ we get
		\begin{equation*}
	\|{\mathfrak{k}_i(\sigma)-\mathbb{F}_{ii}\mathfrak{k}(\sigma)}\|^ 2=\|{\mathbb{F}_{ii}\mathfrak{k}(\sigma)\mathbb{F}_{jj}}\|^2=\|{\langle \mathbb{F}_{ii}\mathfrak{k}(\sigma)\mathbb{F}_{jj}, \mathbb{F}_{ii}\mathfrak{k}(\sigma)\mathbb{F}_{jj}\rangle}\|=\|{\mathbb{F}_{jj}\mathfrak{K}^{\sigma,\sigma}(\mathbb{F}_{ii})\mathbb{F}_{jj}}\|=0
		\end{equation*}
		and \begin{equation*}
			\|{\mathfrak{k}_i(\sigma)-\mathfrak{k}(\sigma)\mathbb{F}_{ii}}\|^ 2=\|{\mathbb{F}_{jj}\mathfrak{k}(\sigma)\mathbb{F}_{ii}}\|^2=\|{\langle \mathbb{F}_{jj}\mathfrak{k}(\sigma)\mathbb{F}_{ii}, \mathbb{F}_{jj}\mathfrak{k}(\sigma)\mathbb{F}_{ii}\rangle}\|=\|{\mathbb{F}_{ii}\mathfrak{K}^{\sigma,\sigma}(\mathbb{F}_{jj})\mathbb{F}_{ii}}\|=0.
		\end{equation*}
		This shows that
		\begin{equation}
			\mathfrak{k}_i(\sigma)=\mathbb{F}_{ii}\mathfrak{k}(\sigma)=\mathfrak{k}(\sigma)\mathbb{F}_{ii}, i\in I_2  \text{ and  thus }\mathfrak{k}(\sigma)=(\mathbb{F}_{11}+\mathbb{F}_{22})\mathfrak{k}(\sigma)=\mathfrak{k}_1(\sigma)+	\mathfrak{k}_2(\sigma).\label{eq-main-1}
		\end{equation}
		Since $\mathfrak{K}$ is a block kernel over $S$, for $\sigma',\sigma\in S,A\in M_2(\mathcal{A})$ and from Equation \eqref{eq-main-1} we obtain
		\begin{equation*}
			\mathfrak{K}^{\sigma,\sigma'}(A)=\langle \mathfrak{k}(\sigma),A\mathfrak{k}(\sigma')\rangle=\sum_{j,i=1}^2\langle	\mathfrak{k}_i(\sigma),A	\mathfrak{k}_j(\sigma')\rangle=\begin{pmatrix}
				\langle	\mathfrak{k}_1(\sigma),A	\mathfrak{k}_1(\sigma')\rangle_{11}&\langle 	\mathfrak{k}_1(\sigma),A	\mathfrak{k}_2(\sigma') \rangle_{12}\\\langle 	\mathfrak{k}_2(\sigma),A	\mathfrak{k}_1(\sigma')\rangle_{21}&\langle 	\mathfrak{k}_2(\sigma),A	\mathfrak{k}_2(\sigma')\rangle_{22}
			\end{pmatrix}.
		\end{equation*}
		%where $\la a,b\ra_{ij}$ will represent the $(i,j)$\textsuperscript{th} entry of $\la a,b\ra\in M_2(\mathcal{B}).$
		
		Suppose that $F^{(\mathcal{B})}$ is the von Neumann $\mathcal{B}$-module and Hilbert $\mathcal{A}$-$\mathcal{B}$-module (see Remark \ref{obs-new-module}), and suppose that the von Neumann $\mathcal{B}$-modules $ \hat{F}_i^{(\mathcal{B})}$ for $i\in I_2.$ Note that $\hat{F}_i$ has a non-degenerate left module action of $\mathcal{A}$ defined by
		\begin{equation}
			b u:=\begin{pmatrix}
				b &0\\0&b
			\end{pmatrix}u\quad\text{for }  u\in \hat{F}_i,b\in\mathcal{A}, i\in I_2.
		\end{equation}
		From Remark (\ref{obs-new-module}), $\hat{F}_i^{(\mathcal{B})}$ is a Hilbert $\mathcal{A}$-$\mathcal{B}$-module for $i\in I_2.$
		And hence  $F^{(\mathcal{B})}\simeq \hat{F}_1^{(\mathcal{B})}\oplus \hat{F}_2^{(\mathcal{B})}$ (via $[\mathfrak{k}(\sigma)]_F\mapsto[\mathbb{F}_{11}\mathfrak{k}(\sigma)]_{\hat{F}_1}+ [\mathbb{F}_{22}\mathfrak{k}(\sigma)]_{\hat{F}_2}$ for $ \mathfrak{k}(\sigma)\in F$). %as $N=N_1\oplus N_2.$
		For $a\in \mathcal{A}$ and $i\in I_2$ we get
		$$\langle [\mathfrak{k}_i(\sigma)],a[\mathfrak{k}_i(\sigma')]\rangle=\sum_{s,r=1}^2\left\langle \mathbb{F}_{ii}\mathfrak{k}(\sigma), \begin{pmatrix} a&0\\0&a\end{pmatrix}\mathbb{F}_{ii}\mathfrak{k}(\sigma')\right\rangle_{r,s}=\sum_{s,r=1}^2\mathfrak{K}^{\sigma,\sigma'}\left(\mathbb{F}_{ii}\begin{pmatrix} a&0\\0&a\end{pmatrix}\mathbb{F}_{ii}\right)_{r,s}=\mathfrak{K}_i^{\sigma,\sigma'}(a).$$
		Therefore $(\hat{F}_i^{(\mathcal{B})},[\mathfrak{k}_i])$ is a Kolmogorov-representations (need not be minimal) for $\mathfrak{K}_i,i\in I_2$ over $S.$
		Define a map $U:\hat{F}_2^{(\mathcal{B})}\to \hat{F}_1^{(\mathcal{B})}$  by $U[u]=[\mathbb{F}_{12}u]$ for all $u\in\hat{F}_2.$ Let $u,v\in \hat{F}_2$ we have
		\begin{equation*}
			\langle U[u],U[v]\rangle=\sum_{s,r=1}^2\langle \mathbb{F}_{12}u,\mathbb{F}_{12}v\rangle_{r,s}=
			\sum_{s,r=1}^2\langle u,\mathbb{F}_{21}\mathbb{F}_{12}v\rangle_{r,s}=\sum_{s,r=1}^2\langle u, v\rangle_{r,s}=\langle [u],[v]\rangle.
		\end{equation*}
	Let $w\in \hat{F}_1,$ then $\mathbb{F}_{21}w\in \hat{F}_2,$ and hence
		\begin{equation*}
			U[\mathbb{F}_{21}w]=[\mathbb{F}_{12}\mathbb{F}_{21}w]=[\mathbb{F}_{11}w]=[w].
		\end{equation*}
		This implies that $U$ is a unitary. Let $u\in \hat{E}_2$ and $a\in\mathcal{A}$ we have
		\begin{equation*}
			Ua[u]=U[\begin{pmatrix} a &0\\ 0&a\end{pmatrix}u]=[\mathbb{F}_{12}\begin{pmatrix} a &0\\ 0&a\end{pmatrix}u]=[\begin{pmatrix} a &0\\ 0&a\end{pmatrix}\mathbb{F}_{12}u]=a[\mathbb{F}_{12}u]=aU[u].
		\end{equation*}
		Therefore $U:\hat{F}_2^{(\mathcal{B})}\to \hat{F}_1^{(\mathcal{B})}$ is a (adjointable) bilinear unitary. Let $\tilde{E_i}:=\overline{span}^s \mathcal{A} \mathfrak{j}_i(\sigma)\mathcal{B} \subseteq E_i$ and  $\tilde{F_i}=\overline{span}^s\mathcal{A} [\mathfrak{k}_i(\sigma)] \mathcal{B}\subseteq \hat{F}_i^{(\mathcal{B})},$ then $(\tilde{E_i},\mathfrak{j}_i)$ and $(\tilde{F_i},[\mathfrak{k}_i])$ are minimal Kolmogorov-representations for $\mathfrak{K}_i,i\in I_2$ over $S.$ And hence $\tilde{T}_i:\tilde{E}_i\to \tilde{F}_i$ defined by $$\tilde{T_i}(a\mathfrak{j}_i(\sigma)b)=a[\mathfrak{k}_i(\sigma)]b,\quad b\in\mathcal{B},\sigma\in S, a\in\mathcal{A},i\in I_2,$$ extends to a (adjointable) bilinear unitary. Suppose $T_i:E_i\to\hat{F}_i^{(\mathcal{B})}$ is the extension of $\tilde{T_i},$ and observe that $T_i$  is a bilinear partial isometry. Let $V:=T_1^*UT_2,$ for $a\in\mathcal{A}$ and $\sigma',\sigma\in S$ we have
		\begin{align*}
			\langle \mathfrak{j}_1(\sigma), Va \mathfrak{j}_2(\sigma')\rangle &=\langle \mathfrak{j}_1(\sigma), T_1^*UT_2a \mathfrak{j}_2(\sigma') \rangle=\langle T_1 \mathfrak{j}_1(\sigma),UT_2a \mathfrak{j}_2(\sigma')\rangle=\langle \tilde{T}_1\mathfrak{j}_1(\sigma), U\tilde{T}_2 a\mathfrak{j}_2(\sigma')\rangle\\
			&=\langle [\mathfrak{k}_1(\sigma)], Ua[\mathfrak{k}_2(\sigma')]\rangle %=\langle \mathbb{E}_{11}x,\mathbb{E}_{12}a\mathbb{E}_{22}x\rangle
			=\sum_{s,r=1}^2\left\langle \mathbb{F}_{11}\mathfrak{k}(\sigma),\mathbb{F}_{12}\begin{pmatrix}
				a &0\\0 &a
			\end{pmatrix}\mathbb{F}_{22}\mathfrak{k}(\sigma')\right\rangle _{r,s}\\
			&=\sum_{s,r=1}^2\left\langle \mathfrak{k}(\sigma), \begin{pmatrix}
				0& a\\0&0
			\end{pmatrix}\mathfrak{k}(\sigma')\right\rangle_{r,s}=\sum_{s,r=1}^2\mathfrak{K}^{\sigma,\sigma'}\begin{pmatrix}
				0&a\\0& 0
			\end{pmatrix}_{r,s}=\sum_{s,r=1}^2\begin{pmatrix}
				0&\mathfrak{L}^{\sigma,\sigma'}(a)\\0& 0
			\end{pmatrix}_{r,s}=\mathfrak{L}^{\sigma,\sigma'}(a).
		\end{align*}
		This shows that $\mathfrak{L}^{\sigma,\sigma'}(a)=\langle \mathfrak{j}_1(\sigma), Va \mathfrak{j}_2(\sigma')\rangle$ for every $a\in \mathcal{A}$ and $\sigma',\sigma\in S.$ 
	\end{proof}
	
	\begin{remark}
		(Uniqueness for $V$). In the setting of Theorem \ref{main-single}, let $V',V:E_2\to E_1 $ be two bilinear adjointable contractions with $\mathfrak{L}^{\sigma,\sigma'}(a)=\langle \mathfrak{j}_1(\sigma), Va \mathfrak{j}_2(\sigma')\rangle=\langle \mathfrak{j}_1(\sigma), V'a \mathfrak{j}_2(\sigma')\rangle$ for all $a\in \mathcal{A}$ and $\sigma',\sigma\in S,$ then 
		\begin{align*}
			\langle a_1 \mathfrak{j}_1(\sigma) b_1, V(a_2 \mathfrak{j}_2(\sigma') b_2)\rangle &= b_1^* \langle  \mathfrak{j}_1(\sigma) , V((a_1^*a_2)\mathfrak{j}_2(\sigma'))\rangle b_2=  b_1^* \langle  \mathfrak{j}_1(\sigma) , V'((a_1^*a_2) \mathfrak{j}_2(\sigma'))\rangle b_2\\&= \langle a_1 \mathfrak{j}_1(\sigma) b_1, V'(a_2 \mathfrak{j}_2(\sigma') b_2)\rangle
		\end{align*} for all $b_1,b_2\in \mathcal{B},a_1,a_2\in \mathcal{A},\sigma',\sigma\in S.$ Thus $P_{\widetilde{E}_1}VP_{\widetilde{E}_2}=P_{\widetilde{E}_1}V'P_{\widetilde{E}_2}$ where $P_{\widetilde{E}_i}:E_i\to E_i$ is the projection of $E_i$
		onto $\widetilde{E}_i$. This implies that contraction $V$ as in Theorem \ref{main-single} is unique if $E_i$'s are minimal Kolmogorov-modules.
	\end{remark}

	The following corollary is a generalization of \cite[Corollary 3.9]{BK20} and \cite[Corollary 2.7]{PC85}.
	
	\begin{corollary}
		Assume that $\mathcal{A}$ is a unital $C^*$-algebra. Let $ \mathfrak{K}_i:S\times S\to B(\mathcal{A},B(\mathcal{H}))$ be a CPD-kernel with minimal Kolmogorov Stinespring representation $({K}_i, \rho_i, L_{\mathfrak{j}_{i}(\sigma)})$ for $i\in I_2$ over a set $S.$ Suppose that $\mathfrak{K}=\begin{pmatrix}
			\mathfrak{K}_1 & \mathfrak{L}\\\mathfrak{L}^*& \mathfrak{K}_2
		\end{pmatrix}:S\times S\to B(M_2(\mathcal{A}), M_2(B(\mathcal{H})))$ is a block CPD-kernel for some  CB-kernel $\mathfrak{L}:S\times S\to B(\mathcal{A},B(\mathcal{H})).$ Then there exists a contraction (unique) $V:{K}_2\to {K}_1$ such that $V\rho_2(a)=\rho_1(a)V$ and $\mathfrak{L}^{\sigma,\sigma'}(a)=L_{\mathfrak{j}_{1}(\sigma)}^*V\rho_2(a)L_{\mathfrak{j}_{2}(\sigma')}$ for $a\in \mathcal{A}$ and $\sigma,\sigma' \in S.$
	\end{corollary}
	\begin{proof}
		Since $({K}_i, \rho_i, L_{\mathfrak{j}_{i}(\sigma)})$ is the minimal Kolmogorov Stinespring representation for $\mathfrak{K}_i, i\in I_2.$ Suppose that $(E_i,L_{\mathfrak{j}_{i}(\sigma)})$ is the (minimal) Kolmogorov-representation for $\mathfrak{K}_i,i\in I_2$ (see Preliminaries subsection) with $E_i = \overline{span}^s \rho_i(\mathcal{A})L_{\mathfrak{j}_{i}(\sigma)} B(\mathcal{H}) \subseteq B(\mathcal{H}, K_i),i\in I_2.$ So from Theorem \ref{main-single}, there is a bilinear adjointable contraction $\widehat{V}:E_2\to E_1$ with \begin{equation}\label{ABC1}
			\mathfrak{L}^{\sigma,\sigma'}(a)=\langle L_{\mathfrak{j}_{1}(\sigma)},\widehat{V}\rho_2(a)L_{\mathfrak{j}_{2}(\sigma')} \rangle=L_{\mathfrak{j}_{1}(\sigma)}^*\widehat{V}\rho_2(a)L_{\mathfrak{j}_{2}(\sigma')}\quad for \quad\quad a\in \mathcal{A}\quad \mbox{and} \quad \sigma',\sigma\in S.
		\end{equation}  
		Since $K_i=\overline{span}\{\rho_i{(a)}L_{\mathfrak{j}_{i}(\sigma)}g : \sigma\in S, g\in \mathcal{H},a\in \mathcal{A} \}$ for $i\in I_2.$ Define a map ${V}:{K}_2\to {K}_1$ by $$V(\rho_2(a)L_{\mathfrak{j}_{2}(\sigma)}g)=(\widehat{V}(\rho_2(a)L_{\mathfrak{j}_{2}(\sigma)}))g\quad for \quad\quad a\in \mathcal{A},\sigma\in S,g\in \mathcal{H}.$$  As $\widehat{V}$ is contraction and right $B(\mathcal{H})$-linear, for $a\in \mathcal{A},g\in \mathcal{H}$ and $\sigma\in S$ we get
		\begin{align*}
			{\big{\| \big|}(\widehat{V}(\rho_2(a)L_{\mathfrak{j}_{2}(\sigma)}))g\rangle \langle g\big{| \big\|}}&={\big\|\widehat{V}(\rho_2(a)L_{\mathfrak{j}_{2}(\sigma)})|g\rangle \langle g| \big\|}={\big\|\widehat{V}(\rho_2(a)L_{\mathfrak{j}_{2}(\sigma)}|g\rangle \langle g|)\big\|}\\&\le \big\|\rho_2(a)L_{\mathfrak{j}_{2}(\sigma)}|g\rangle \langle g| \big\|={\big{\| \big|}(\rho_2(a)L_{\mathfrak{j}_{2}(\sigma)})g\rangle \langle g\big{| \big\|}}.
		\end{align*} This shows that $||(\widehat{V}(\rho_2(a)L_{\mathfrak{j}_{2}(\sigma)}))g||\le ||\rho_2(a)L_{\mathfrak{j}_{2}(\sigma)}g||$ for $a\in \mathcal{A},\sigma\in S,g\in \mathcal{H},$ and hence $V$ is a contraction. Since $\widehat{V}$ is left $\mathcal{A}$-linear, for $\sigma\in S,a,a'\in \mathcal{A}$ and $g\in \mathcal{H}$ we obtain 
		\begin{align*}
			V\rho_2(a)(\rho_2(a')L_{\mathfrak{j}_{2}(\sigma)}g)&=V(\rho_2(aa')L_{\mathfrak{j}_{2}(\sigma)}g)=\widehat{V}(\rho_2(aa')L_{\mathfrak{j}_{2}(\sigma)})g=\widehat{V}(\rho_2(a)\rho_2(a')L_{\mathfrak{j}_{2}(\sigma)})g\\&=\rho_1(a)\widehat{V}(\rho_2(a')L_{\mathfrak{j}_{2}(\sigma)})g=\rho_1(a)V(\rho_2(a')L_{\mathfrak{j}_{2}(\sigma)}g).
		\end{align*} This implies that $\rho_1(a)V=V\rho_2(a)$ for $a\in \mathcal{A}.$ From Equation (\ref{ABC1}), $\mathfrak{L}^{\sigma,\sigma'}(a)g=L_{\mathfrak{j}_{1}(\sigma)}^*V\rho_2(a)L_{\mathfrak{j}_{2}(\sigma')}g$ for $g\in 
		\mathcal{H},\sigma',\sigma\in S.$ For the uniqueness of $V,$ let $V'$ be the another contraction such that $\rho_1(a)V'=V'\rho_2(a)$ and $\mathfrak{L}^{\sigma,\sigma'}(a)=L_{\mathfrak{j}_{1}(\sigma)}^*V'\rho_2(a)L_{\mathfrak{j}_{2}(\sigma')}$ for every $\sigma',\sigma\in S$ and $a\in \mathcal{A}.$ For $\sigma',\sigma\in S,g,h\in \mathcal{H}$ and $b,a\in \mathcal{A}$ we have  
		\begin{align*}
			\langle \rho_1(a)L_{\mathfrak{j}_{1}(\sigma)}h,V\rho_2(b)L_{\mathfrak{j}_{2}(\sigma')}g \rangle&=\langle h,L_{\mathfrak{j}_{1}(\sigma)}^*V\rho_2(a^*b)L_{\mathfrak{j}_{2}(\sigma')}g \rangle=\langle h,\mathfrak{L}^{\sigma,\sigma'}(a^*b)g \rangle\\&=\langle h,L_{\mathfrak{j}_{1}(\sigma)}^*V'\rho_2(a^*b)L_{\mathfrak{j}_{2}(\sigma')}g \rangle=	\langle \rho_1(a)L_{\mathfrak{j}_{1}(\sigma)}h,V'\rho_2(b)L_{\mathfrak{j}_{2}(\sigma')}g \rangle.
		\end{align*}
		This implies that the uniqueness of $V.$
	\end{proof}

	\section{Semigroups of Block completely positive definite kernels}

The goal of this section is to study the structure of a semigroup of CPD-kernels along the lines of Theorem \ref{main-single}. Denote by $\mathbb{T}$ the semigroup $\mathbb{N}_0$ or $\mathbb{R}_+$ under addition. % but our real interest lies in $\mathbb{T}=\mathbb{R}_+$. 
First we start with the following definitions.
	\begin{definition}
		Let $S$ be a set and $\mathcal{B}$ be a unital $C^*$-algebra. A family $\mathfrak{K}=(\mathfrak{K}_t)_{t\in \mathbb{T}}$ of CPD-kernels over $S$ on $\mathcal{B}$ is called {\rm one-parameter CPD-semigroup} or {\rm quantum dynamical semigroup} (QDS) over $S$ \cite{BBLS04,Sk01} if
		\begin{itemize}
			\item[(1)] $\mathfrak{K}_{t}^{\sigma,\sigma'} \circ \mathfrak{K}_s^{\sigma,\sigma'}=\mathfrak{K}_{t+s}^{\sigma,\sigma'}$ for $\sigma',\sigma\in S$ and $t,s\in \mathbb{T}.$
			\item[(2)] $\mathfrak{K}_{0}^{\sigma,\sigma'}(w)=w$ for  $\sigma',\sigma\in S$ and $w\in \mathcal{B}.$
			\item[(3)] $\mathfrak{K}_{t}^{\sigma,\sigma'}({\bf{1}})\le{\bf{1}}$ for $t\in \mathbb{T}$ and $\sigma',\sigma\in S.$
		\end{itemize}
	\end{definition}
	It is called {\it quantum Markov semigroup} (or QMS) over $S$ \cite{BBLS04,AK01} if $\mathfrak{K}_s$ is unital over $S$ for every $s\in \mathbb{T}.$ 
	%In practice, in addition to $(1)$-$(3)$ we may assume continuity of $t\to \mathfrak{K}_{t}^{\sigma,\sigma'}(a)$ in different topologies, depending upon the context.
	
	\begin{definition}\label{DS1}
		Suppose $\mathcal{B}$ is a $C^*$-algebra. Let $F = (F_t)_{t\in \mathbb{T}}$ be a family of Hilbert $\mathcal{B}$-$\mathcal{B}$-modules such that $F_0 = \mathcal{B}$ and $\alpha= (\alpha_{s,t})_{t,s\in \mathbb{T}}$ be a family of bilinear (adjointable) isometries $\alpha_{s,t}:F_{s+t}\to F_s \odot F_t$ with
		\begin{equation}\label{DS2}
			(I_{F_t}\odot \alpha_{s,r})\alpha_{t,s+r}=(\alpha_{t,s}\odot I_{F_r})\alpha_{t+s,r}\quad \mbox{for} \quad {t,r,s\in \mathbb{T}},
		\end{equation}  then we say that $(F, \alpha)$ is an {\rm inclusion system}. It is called {\rm product system} if $\alpha_{s,t}$ is a unitary for ${s,t\in \mathbb{T}}.$ 	Further, assume that $\mathcal{B}$ is a vNa, then we consider the inclusion system of von Neumann $\mathcal{B}$-$\mathcal{B}$-modules.
	\end{definition}

	\begin{definition}
		Suppose that $(F, \alpha)$ is an inclusion system. Let $\eta^{\sigma \odot }=(\eta_s^{\sigma})_{s\in \mathbb{T}}$ be a family of vectors $\eta_s^{\sigma}\in F_s$ is said to be {\rm unit} of $(F, \alpha)$ if $\alpha_{s,t}(\eta_{s+t}^{\sigma})=\eta_s^{\sigma}\odot \eta_t^{\sigma}$ for $s,t\in \mathbb{T}.$ It is called {\rm generating} if $\eta_s^{\sigma}$ is cyclic in $F_s$ for every $s\in \mathbb{T},$ and {\rm unital} if $\langle \eta_s^{\sigma},\eta_s^{\sigma}\rangle=1$
		for ${s\in \mathbb{T}}$ and $\sigma\in S.$ Further, assume that $(F, \alpha)$ is a product system,
		a unit $\eta^{\sigma \odot}=(\eta_s^{\sigma})_{s\in \mathbb{T}}$ is called {\rm generating} unit for the product system $(F, \alpha)$ if $F_s$ is spanned
		by images of vectors $b_n\eta_{s_n}^{\sigma} \odot \dots \odot b_1\eta_{s_1}^{\sigma}b_0$ $( \sum s_i = s,s_i\in \mathbb{T},b_i\in \mathcal{B},\sigma\in S)$ under the successive applications of the appropriate mappings $I \odot \alpha^*_{t,t'}\odot I.$
	\end{definition}
	\begin{remark}\label{comp}
		Let $\mathfrak{K}:S\times S\to B(\mathcal{A},\mathcal{B})$ and $ \mathfrak{L}:S\times S\to B(\mathcal{B},\mathcal{C})$ be two CPD-kernels over a set $S$ with the Kolmogorov-representations $(F, \mathfrak{i})$ and $(E,\mathfrak{j}),$ respectively. Suppose that $(K, \mathfrak{k})$ is the Kolmogorov-representation for $\mathfrak{L}\circ\mathfrak{K}.$ For $\sigma',\sigma\in S$ and $a\in\mathcal{A}$ we have 
		\begin{equation}\label{DPS1}
			\langle \mathfrak{i}(\sigma)\odot \mathfrak{j}(\sigma), a\mathfrak{i}(\sigma')\odot \mathfrak{j}(\sigma')\rangle=\langle \mathfrak{j}(\sigma),\langle \mathfrak{i}(\sigma), a\mathfrak{i}(\sigma')\rangle \mathfrak{j}(\sigma')\rangle=\langle \mathfrak{j}(\sigma),\mathfrak{K}^{\sigma,\sigma'}(a)\mathfrak{j}(\sigma')\rangle=(\mathfrak{L}\circ\mathfrak{K})^{\sigma,\sigma'}(a).
		\end{equation}
		This implies that $(F\odot E ,\mathfrak{i}\odot \mathfrak{j})$ is the Kolmogorov-representation (need not be minimal) for $\mathfrak{L}\circ\mathfrak{K}$ over $S.$ The mapping
		\begin{equation}\label{eq-iso}
			\mathfrak{k}(\sigma)\mapsto \mathfrak{i}(\sigma)\odot \mathfrak{j}(\sigma)\quad \quad \mbox{for} \quad \sigma\in S,
		\end{equation}
		extends as an isometry $\mu:K\to F\odot E.$ Consider $K$ as the submodule  $\overline{span}(\mathcal{A}\mathfrak{i}(\sigma)\odot\mathfrak{j}(\sigma)\mathcal{C})$ of $F\odot E.$
		Note that $\overline{span}(\mathcal{A}\mathfrak{i}(\sigma)\odot\mathcal{B}\mathfrak{j}(\sigma)\mathcal{C})=\overline{span}(\mathcal{A}\mathfrak{i}(\sigma)\mathcal{B}\odot\mathfrak{j}(\sigma)\mathcal{C})=\overline{span}(\mathcal{A}\mathfrak{i}(\sigma)\mathcal{B}\odot\mathcal{B}\mathfrak{j}(\sigma)\mathcal{C})=F\odot E.$
	\end{remark}
	
	Let $(F, \alpha)$ be an inclusion system with unit $\eta^{\sigma \odot}.$ Define a map $\mathfrak{K}_t:S\times S \to B(\mathcal{B})$ by
	$\mathfrak{K}_t^{\sigma,\sigma'}(b)=\langle \eta_t^{\sigma},b\eta_t^{\sigma'}\rangle$ for $b\in \mathcal{B},\sigma',\sigma\in S$ and $t\in \mathbb{T},$ then
	\begin{align*}
		(\mathfrak{K}_t \circ \mathfrak{K}_s)^{\sigma,\sigma'}(b)&=\mathfrak{K}_t^{\sigma,\sigma'}(\langle \eta_s^{\sigma},b\eta_s^{\sigma'}\rangle)=\langle \eta_t^{\sigma}, \langle \eta_s^{\sigma},b\eta_s^{\sigma'}\rangle\eta_t^{\sigma'}\rangle= \langle \eta_s^{\sigma}\odot \eta_t^{\sigma},b(\eta_s^{\sigma'}\odot \eta_t^{\sigma'})\rangle=\langle \eta_{t+s}^{\sigma},b\eta_{t+s}^{\sigma'}\rangle=\mathfrak{K}_{t+s}^{\sigma,\sigma'}(b).
	\end{align*}
	This shows that $(\mathfrak{K}_t)_{t\in \mathbb{T}}$ is QDS. Further, suppose that  $\eta^{\sigma \odot}$ is unital, then $(\mathfrak{K}_t)_{t\in \mathbb{T}}$ is QMS . The following remark in the converse direction:
	
	\begin{remark}\label{DS3}
		Suppose that $\mathfrak{K}=(\mathfrak{K}_s)_{s\in \mathbb{T}}$ is a QDS over $S$ on a unital $C^*$-algebra $\mathcal{B}$ with the (minimal)
		Kolmogorov-representation $(F_s, \eta_s)$ for $\mathfrak{K}_s$ over $S.$ (That is, $\eta_s^{\sigma}$ ($\eta_s:\sigma \mapsto\eta_s^{\sigma}$) is a cyclic vector in $F_s$ with $\mathfrak{K}_s^{\sigma,\sigma'}(b)=\langle \eta_s^{\sigma},b\eta_s^{\sigma'}\rangle$ for $s\in \mathbb{T},b\in \mathcal{B}$ and $\sigma',\sigma\in S,$ and $\eta_0^{\sigma}=1$, $F_0 = \mathcal{B}$).  Define a map $\alpha_{s,t}:F_{s+t}\to F_s \odot F_t$ by
		\begin{equation}
			\eta^{\sigma}_{t+s}\mapsto \eta_s^{\sigma}\odot \eta_t^{\sigma}.
		\end{equation}
		From Remark \ref{comp}, $\alpha_{s,t}$’s are bilinear isometries and it satisfies
		\begin{align*}
			(I_{F_r}\odot\alpha_{s,t})\alpha_{r,s+t}(\eta^{\sigma}_{r+s+t})&=(I_{F_r}\odot\alpha_{s,t})(\eta^{\sigma}_{r}\odot \eta^{\sigma}_{s+t})=\eta^{\sigma}_{r}\odot (\eta^{\sigma}_{s}\odot \eta^{\sigma}_{t})\\&=(\eta^{\sigma}_{r}\odot\eta^{\sigma}_{s})\odot \eta^{\sigma}_{t}
			=(\alpha_{r,s}\odot I_{F_t})(\eta^{\sigma}_{r+s}\odot \eta^{\sigma}_{t})=(\alpha_{r,s}\odot I_{F_t})\alpha_{r+s,t}(\eta^{\sigma}_{r+s+t})
		\end{align*}

		This implies that $(F= (F_t), \alpha = (\alpha_{s,t}))$ is an inclusion system with generating unit $\eta^{\sigma \odot }=(\eta_t^{\sigma}).$ Consider $F_{s+t}$ as the submodule  $\overline{span}(\mathcal{B}\eta_s^{\sigma}\odot\eta_t^{\sigma}\mathcal{B})$ of $F_s\odot F_t.$ Moreover, for $t\in \mathbb{T},$ let each $\mathfrak{K}_t$ be a normal CPD-kernel over a given set $S$ on a vNa $\mathcal{B},$ then the Kolmogorov-module
		$F_t$ is a von Neumann $\mathcal{B}$-$\mathcal{B}$-module, and hence $(F, \alpha)$ is an inclusion system with generating unit $\eta^{\sigma \odot }=(\eta_t^{\sigma}).$ 
	\end{remark}
	
Suppose that $\mathfrak{K}=(\mathfrak{K}_s)_{s\ge0}$ is a QDS over $S$ on $\mathcal{B},$ then we say that the inclusion system $(F, \alpha, \eta^{\sigma \odot })$ in Remark \ref{DS3} is the {\it inclusion system associated to $\mathfrak{K}$}. Sometimes we use $(F, \eta^{\sigma \odot })$ for $(F, \alpha, \eta^{\sigma \odot }).$

	\begin{definition}
		Suppose that $(F, \alpha)$ and $(E, \beta)$ are inclusion systems and $V = (V_s)_{s\in \mathbb{T}}$ is a family of bilinear adjointable operators $V_s : E_s\to F_s$ such that $\|V_s\|\le e^{tm}$ for some $m\in \mathbb{R}.$ We say that $V$ is a {\rm weak morphism} or {\rm morphism} if \begin{equation}
			V_{s+t}=\alpha_{s,t}^*(V_s\odot V_t)\beta_{s,t}
		\end{equation} for every $t,s\in \mathbb{T}$ and $\alpha_{s,t}$ is adjointable. It is called {\rm strong morphism} if \begin{equation}
			\alpha_{s,t}V_{s+t}=(V_s\odot V_t)\beta_{s,t}\quad \quad for \quad t,s\in \mathbb{T}.
		\end{equation} 
	\end{definition}
	
	\begin{example}
		Let $\mathcal{B}$ be a $C^*$-algebra. Suppose that $(a_s^{\sigma})_{s\ge 0}$ and $(b_s^{\sigma})_{s\ge 0},$ $\sigma\in S$ are semigroups on $\mathcal{B}$ (that is, $a_t^{\sigma}a_s^{\sigma}=a_{t+s}^{\sigma}$) and $(\mathfrak{K}_{s,i})_{s\ge 0},$ $i\in I_2$ is QDS over $S$ on $\mathcal{B}$ such that $\mathfrak{K}^{\sigma,\sigma'}_{s,1}(\cdot)-a_s^{\sigma}(\cdot)a_s^{\sigma' {^*}}$ and $\mathfrak{K}^{\sigma,\sigma'}_{s,2}(\cdot)-b_s^{\sigma}(\cdot)b_s^{\sigma' {^*}}$ are CPD-kernels over $S$ for $s\ge 0.$ Define a map $\tau_s:S\times S\to B(M_2(\mathcal{B}))$ by 	\begin{equation*}
			\tau_s^{\sigma,\sigma'} \begin{pmatrix}
				a_{11} &a_{12}\\a_{21}&a_{22}
			\end{pmatrix}
			= \begin{pmatrix}
				\mathfrak{K}^{\sigma,\sigma'}_{s,1}(a_{11}) & a_s^{\sigma} a_{12} b_s^{\sigma' {^*}}  \\\ b_s^{\sigma}a_{21}a_s^{\sigma' {^*}} &\mathfrak{K}^{\sigma,\sigma'}_{s,2}(a_{22})
			\end{pmatrix},
		\end{equation*}
		then $\tau_s$ is CPD-kernel over $S$ and
		\begin{equation*}
			\tau_s^{\sigma,\sigma'} \begin{pmatrix}
					a_{11} &a_{12}\\a_{21}&a_{22}
			\end{pmatrix}
			= \begin{pmatrix}
				a_s^{\sigma}& 0\\ 0& b_s^{\sigma}
			\end{pmatrix} \begin{pmatrix}
				a_{11} &a_{12}\\a_{21}&a_{22}
			\end{pmatrix} \begin{pmatrix}
				a_s^{\sigma' {^*}}& 0\\ 0& b_s^{\sigma' {^*}}
			\end{pmatrix} +\begin{pmatrix}
				\mathfrak{K}^{\sigma,\sigma'}_{s,1}(a_{11})-a_s^{\sigma}a_{11}a_s^{\sigma' {^*}}&0\\0& \mathfrak{K}^{\sigma,\sigma'}_{s,2}(a_{22})-b_s^{\sigma}a_{22}b_s^{\sigma' {^*}}
			\end{pmatrix}.
		\end{equation*} Thus $(\tau_s)_{s\ge 0}$ is block QDS.
	\end{example}
	
	\begin{lemma}
		Let $(F^i, \alpha^i, \eta_i^{\sigma \odot})$ be an inclusion systems associated to CPD-kernel semigroup $\mathfrak{K}_i=(\mathfrak{K}_{t,i})_{t\in \mathbb{T}},$  $i\in I_2$ over a given set $S$ on a unital $C^*$-algebra $\mathcal{B}.$ If $V = (V_t):F^2 \to F^1$ is a  contractive morphism, then there is a block CPD-kernel semigroup $\mathfrak{T}=(\mathfrak{T}_t)_{t\ge 0}$ over $S$ on $M_2(\mathcal{B})$ such that $\mathfrak{T}_t=\begin{pmatrix}
			\mathfrak{K}_{t,1} & \mathfrak{L}_t\\\mathfrak{L}_t^*& \mathfrak{K}_{t,2}
		\end{pmatrix}: S\times S\to  B( M_2(\mathcal{B}))$ and $\mathfrak{L}_t^{\sigma,\sigma'}(b)=\langle \eta_{t,1}^{\sigma},V_t b\eta_{t,2}^{\sigma'}\rangle$ for $b\in \mathcal{B}$ and $\sigma',\sigma\in S.$
	\end{lemma}
	\begin{proof}
		Define block maps $\mathfrak{T}_t: S\times S\to  B( M_2(\mathcal{B}))$ by $\mathfrak{T}_t=\begin{pmatrix}
			\mathfrak{K}_{t,1} & \mathfrak{L}_t\\\mathfrak{L}_t^*& \mathfrak{K}_{t,2}
		\end{pmatrix},$ where $\mathfrak{L}_t^{\sigma,\sigma'}(b)=\langle \eta_{t,1}^{\sigma},V_t b\eta_{t,2}^{\sigma'}\rangle$ for $b\in \mathcal{B}$ and $\sigma',\sigma\in S.$
		Since $V_t:F_t^2 \to F_t^1$ is bilinear adjointable contraction, we get $\mathfrak{T}_t$ is block CPD-kernel over $S$ for every $t\ge 0$ (see Lemma \ref{lem-single}). Therefore, it is enough to prove that $\mathfrak{T}=(\mathfrak{T}_t)_{t\ge 0}$ is a semigroup over $S$ on $M_2(\mathcal{B}).$ Let $\begin{pmatrix}
			a_{11} &a_{12}\\a_{21}&a_{22}
		\end{pmatrix}\in M_2(\mathcal{B})$ and $\sigma',\sigma\in S$ we have 
		\begin{align*}
			(\mathfrak{T}_s \circ	\mathfrak{T}_t)^{\sigma,\sigma'} \begin{pmatrix}
				a_{11} &a_{12}\\a_{21}&a_{22}
			\end{pmatrix}= \mathfrak{T}_s^{\sigma,\sigma'}\begin{pmatrix}
				\mathfrak{K}^{\sigma,\sigma'}_{t,1}(a_{11}) & \mathfrak{L}^{\sigma,\sigma'}_t(a_{12})\\\mathfrak{L}_t^{{\sigma,\sigma'}^*}(a_{21})& \mathfrak{K}^{\sigma,\sigma'}_{t,2}(a_{22})
			\end{pmatrix}=\begin{pmatrix}
				\mathfrak{K}^{\sigma,\sigma'}_{s+t,1}(a_{11}) & \mathfrak{L}^{\sigma,\sigma'}_s(\mathfrak{L}^{\sigma,\sigma'}_t(a_{12}))\\\mathfrak{L}_s^{{\sigma,\sigma'}^*}(\mathfrak{L}_t^{{\sigma,\sigma'}^*}(a_{21}))& \mathfrak{K}^{\sigma,\sigma'}_{s+t,2}(a_{22})
			\end{pmatrix}.
		\end{align*}
		If $V$ is a morphism, then $(\mathfrak{L}_s)_{s\ge 0}$ is a semigroup, and hence $\mathfrak{T}=(\mathfrak{T}_s)_{s\ge 0}$ is a semigroup.
	\end{proof}
	
	Now we present the semigroup version of Theorem \ref{main-single}.
	\begin{theorem}
		Suppose that $\mathfrak{T}=(\mathfrak{T}_t)_{t\ge 0}$ is a semigroup of block normal CPD-kernels over a set $S$ on $M_2(\mathcal{B}),$ where $\mathcal{B}$ is a vNa and $\mathfrak{T}_t=\begin{pmatrix}
			\mathfrak{K}_{t,1} & \mathfrak{L}_t\\\mathfrak{L}_t^*& \mathfrak{K}_{t,2}
		\end{pmatrix}.$ Then there exist an inclusion systems $(F^i, \alpha^i, \xi_i^{\sigma \odot})$
		associated to $\mathfrak{K}_i=(\mathfrak{K}_{t,i})_{t\in \mathbb{T}},$ $i\in I_2$ and a contractive unique morphism $V = (V_t):F^2 \to F^1$ with $\mathfrak{L}_t^{\sigma,\sigma'}(b)=\langle \xi_{t,1}^{\sigma},V_t b\xi_{t,2}^{\sigma'}\rangle$ for every $\sigma',\sigma\in S$ and $b\in \mathcal{B}.$
	\end{theorem}
	\begin{proof}
		We will prove it by applying the same concepts of Theorem \ref{main-single} to the semigroup version. Suppose that $(F= (F_t), \alpha = (\alpha_{t,s}), \eta^{\sigma\odot} = (\eta_t^{\sigma}))$ is an inclusion system associated to $\mathfrak{T}.$ Set $\mathbb{F}_{ij}:=\mathbf{1} \otimes {\bf e}_{ij}$  in $ \mathcal{B}\otimes M_2,$ where $\{{\bf e}_{ij}\}$'s denote the units of $M_2.$ Let $\hat{F}_t^i:=\mathbb{F}_{ii} F_t\subset F_t,i\in I_2,$ then $\hat{F}_t^i$’s are $M_2(\mathcal{B})$-submodules of $F_t$ with SOT closed and $F_t=\hat{F}_t^1\oplus \hat{F}_t^2$ for every $t\ge0.$ For $\sigma\in S,$ let $\eta_{t,i}^{\sigma}=\mathbb{F}_{ii}\eta_t^{\sigma} \mathbb{F}_{ii}\in \hat{F}_t^i, i\in I_2,$
		then (as in Theorem \ref{main-single})
		\begin{equation}\label{1}
			\eta_t^{\sigma} =\eta_{t,1}^{\sigma} +\eta_{t,2}^{\sigma}  \quad with \quad \langle \eta_{t,1}^{\sigma},\eta_{t,2}^{\sigma}\rangle=0 \quad and \quad \eta_{t,i}^{\sigma}=\mathbb{F}_{ii}\eta_t^{\sigma}=\eta_t^{\sigma} \mathbb{F}_{ii}
		\end{equation} for all $t\ge0,i\in I_2.$ Since $\alpha_{t,s}:F_{t+s}\to F_t \odot F_s$ are the canonical maps 
		$\eta_{t+s}^{\sigma} \mapsto \eta_t^{\sigma}\odot \eta_s^{\sigma},$ from Equation (\ref{1}) we get
		\begin{equation}\label{2}
			\alpha_{t,s}(\eta_{t+s,i}^{\sigma})=\alpha_{t,s}(\mathbb{F}_{ii}\eta_{t+s}^{\sigma}\mathbb{F}_{ii})=\mathbb{F}_{ii}\eta_{t}^{\sigma}\odot \eta_{s}^{\sigma}\mathbb{F}_{ii}=\eta_{t,i}^{\sigma}\odot \eta_{s,i}^{\sigma} \quad \quad\mbox{for} \quad \quad s,t\ge0,\sigma\in S,i\in I_2.
		\end{equation}
		Suppose that $ \hat{F}_t^{i(\mathcal{B})}$ is the von Neumann $\mathcal{B}$-modules for $i\in I_2$ (see Remark \ref{obs-new-module}) and $F_t^{(\mathcal{B})}$ is the von Neumann $\mathcal{B}$-$\mathcal{B}$-modules. Also, $\hat{F}_t^{i(\mathcal{B})}$ is a von Neumann $\mathcal{B}$-$\mathcal{B}$-module, where the left module action of $\mathcal{B}$ defined as
		\begin{equation*}
			b [u]:=[\begin{pmatrix}
				b &0\\0&b
			\end{pmatrix}u]\quad\text{for }  u\in \hat{F}_t^i,b\in\mathcal{B}, i\in I_2,
		\end{equation*} then $F_t^{(\mathcal{B})}\simeq \hat{F}_t^{1(\mathcal{B})}\oplus \hat{F}_t^{2(\mathcal{B})}$ for $t\ge 0$ (see Theorem \ref{main-single}). Let $\xi_{t,i}^{\sigma}=[\eta_{t,i}^{\sigma}]\in F_t^{i(\mathcal{B})},b\in \mathcal{B},i\in I_2$ and $\sigma',\sigma\in S,$ we have
		$$\langle \xi_{t,i}^{\sigma},b\xi_{t,i}^{\sigma'}\rangle=\sum_{s,r=1}^2\left\langle \mathbb{F}_{ii}\eta_{t}^{\sigma}, \begin{pmatrix} b&0\\0&b \end{pmatrix}\mathbb{F}_{ii}\eta_{t}^{\sigma'}\right\rangle_{r,s}=\sum_{s,r=1}^2\mathfrak{T}_t^{\sigma,\sigma'}\left(\mathbb{F}_{ii}\begin{pmatrix} b&0\\0&b\end{pmatrix}\mathbb{F}_{ii}\right)_{r,s}=\mathfrak{K}_{t,i}^{\sigma,\sigma'}(b).$$
		This implies that $(\hat{F}_t^{i(\mathcal{B})}, \xi_{t,i})$ is a (need not be minimal) Kolmogorov-representation for $\mathfrak{K}_{t,i},$  $i\in I_2$ over $S.$ Let $F_t^i=\overline{span}^s\mathcal{B} \xi_{t,i}^{\sigma} \mathcal{B}\subseteq \hat{F}_t^{i(\mathcal{B})}$ be the minimal GNS-module for $\mathfrak{K}_{t,i},$  $i\in I_2.$ Let $\alpha_{s,t}^i:F_{t+s}^i\to F_t^i \odot F_s^i$ be
		the canonical maps (see Remark \ref{DS3}) defined by\begin{equation*}
			\xi^{\sigma}_{t+s,i}\mapsto \xi_{t,i}^{\sigma}\odot \xi_{s,i}^{\sigma} \quad \quad for \quad s,t\ge 0,\sigma\in S,i\in I_2.
		\end{equation*}
		Therefore $(F^i = (F_t^i), \alpha^i = (\alpha_{t,s}^i), \xi_i^{\sigma \odot} = (\xi_{t,i}^{\sigma}))$ is the inclusion system associated to $\mathfrak{K}_{i},$  $i \in I_2.$
		( From the inclusion system associated to $\mathfrak{T}$ and Equation (\ref{2}), we get the inclusion systems associated to $\mathfrak{K}_{i}$’s.) For $t\ge 0,$ suppose that $T_t^i:F_t^i\to \hat{F}_t^{i(\mathcal{B})}$ is an inclusion maps, and  $U_t:\hat{F}_t^{2(\mathcal{B})}\to \hat{F}_t^{1(\mathcal{B})}$ is given by $$U_t[u]=[\mathbb{F}_{12}u] \quad \mbox{for all}\quad u\in\hat{F}_t^2,$$ 
		then $U_t$’s are bilinear unitaries and $T_t^i$’s are bilinear adjointable isometries (see Theorem \ref{main-single}). 
		Take  $V_t:=T_t^{1*}U_tT_t^2,$ then $V_t: F_t^2\to F_t^1$ is bilinear adjointable contraction. For $b\in \mathcal{B}$ and $\sigma',\sigma\in S$ we have
		\begin{align*}
			\langle \xi_{t,1}^{\sigma}, V_t b \xi_{t,2}^{\sigma'}\rangle &=\langle \xi_{t,1}^{\sigma}, T_t^{1*}U_tT_t^2 b \xi_{t,2}^{\sigma'} \rangle=\langle T_t^{1} [\eta_{t,1}^{\sigma}], U_tT_t^2 b [\eta_{t,2}^{\sigma'}] \rangle=\langle [\eta_{t,1}^{\sigma}], U_t [\begin{pmatrix}
				b &0\\0&b
			\end{pmatrix}\eta_{t,2}^{\sigma'}]\rangle\\
			&=\langle [\eta_{t,1}^{\sigma}], [\begin{pmatrix}
				0 &b\\0&0
			\end{pmatrix}\eta_{t,2}^{\sigma'}]\rangle %=\la \mathbb{E}_{11}x,\mathbb{E}_{12}a\mathbb{E}_{22}x\ra
			=\sum_{s,r=1}^2\left\langle \eta_{t,1}^{\sigma}, \begin{pmatrix}
				0 &b\\0&0
			\end{pmatrix}\eta_{t,2}^{\sigma'}\right\rangle_{r,s}\\
			&=\sum_{s,r=1}^2\left(\mathbb{F}_{11}\mathfrak{T}^{\sigma,\sigma'}_t \begin{pmatrix}
				0& b\\0&0
			\end{pmatrix}\mathbb{F}_{22}\right)_{r,s}=\sum_{s,r=1}^2\begin{pmatrix}
				0&\mathfrak{L}_t^{\sigma,\sigma'}(b)\\0& 0
			\end{pmatrix}_{r,s}=\mathfrak{L}_t^{\sigma,\sigma'}(b).
		\end{align*} It is enough to show that $V=(V_t)$ is a unique (weak) morphism from $(F^2,\alpha^2)$ to $(F^1,\alpha^1).$ For $a,b,c,d\in \mathcal{B}$ and $\sigma',\sigma\in S$ we get
		\begin{align*}
			\langle a\xi_{t+s,1}^{\sigma}b, V_{t+s} (c \xi_{t+s,2}^{\sigma'}d)\rangle &=b^*\mathfrak{L}_{t+s}^{\sigma,\sigma'}(a^*c)d=b^*\mathfrak{L}_{s}^{\sigma,\sigma'}(\mathfrak{L}_{t}^{\sigma,\sigma'}(a^*c))d=b^*\mathfrak{L}_{s}^{\sigma,\sigma'}(\langle \xi_{t,1}^{\sigma},a^*cV_t \xi_{t,2}^{\sigma'}\rangle)d\\&=b^* \langle \xi_{s,1}^{\sigma},V_s(\langle \xi_{t,1}^{\sigma},a^*cV_t \xi_{t,2}^{\sigma'}\rangle) \xi_{s,2}^{\sigma'}\rangle d=b^* \langle \xi_{s,1}^{\sigma},\langle \xi_{t,1}^{\sigma},a^*cV_t \xi_{t,2}^{\sigma'}\rangle V_s \xi_{s,2}^{\sigma'}\rangle d\\&=b^* \langle \xi_{t,1}^{\sigma}\odot\xi_{s,1}^{\sigma},a^*c(V_t \odot V_s )(\xi_{t,2}^{\sigma'} \odot\xi_{s,2}^{\sigma'})\rangle d =\langle \alpha_{t,s}^1(a\xi_{t+s,1}^{\sigma}b),(V_t \odot V_s )\alpha_{t,s}^2(c\xi_{t+s,2}^{\sigma'}d)\rangle \\&=\langle a\xi_{t+s,1}^{\sigma}b,\alpha_{t,s}^{1*}(V_t \odot V_s )\alpha_{t,s}^2(c\xi_{t+s,2}^{\sigma'}d)\rangle.
		\end{align*} This implies that $V= (V_t)_{t\ge0}$ is a morphism.
		For the uniqueness of $V$, let $V'= (V'_t)_{t\ge0}$ be the another morphism such that $\mathfrak{L}_t^{\sigma,\sigma'}(b)=\langle \xi_{t,1}^{\sigma},V'_t b\xi_{t,2}^{\sigma'}\rangle$ for $t\ge 0,b\in \mathcal{B},$ then
		\begin{align*}
			\langle a\xi_{t,1}^{\sigma}b,V_t(c\xi_{t,2}^{\sigma'}d) \rangle =b^* \mathfrak{L}_t^{\sigma,\sigma'}(a^*c)d=\langle a\xi_{t,1}^{\sigma}b,V'_t(c\xi_{t,2}^{\sigma'}d) \rangle
		\end{align*} for all $a,b,c,d\in \mathcal{B}$ and $\sigma',\sigma\in S.$ Thus $V_t=V'_t$ for $t\ge 0.$
	\end{proof}
	
	\begin{example}
		Suppose that $F$ is a von Neumann $M_2(\mathcal{B})$-$M_2(\mathcal{B})$-module, where $\mathcal{B}$ is a vNa. Let $\alpha_{\sigma}=\begin{pmatrix}
			\alpha_{{\sigma},1} &0\\0&\alpha_{\sigma,2}
		\end{pmatrix}\in M_2(\mathcal{B})$
		and $\eta^{\sigma}\in F$ such that $\eta^{\sigma}=\mathbb{F}_{11}\eta^{\sigma}\mathbb{F}_{11} + \mathbb{F}_{22}\eta^{\sigma}\mathbb{F}_{22},$ 
		where $\sigma\in S,\mathbb{F}_{ij}=\mathbf{1} \otimes {\bf e}_{ij}\in \mathcal{B}\otimes M_2.$ 
		Let $\zeta^{{\sigma}\odot}(\alpha_{\sigma},\eta^{\sigma}) = (\zeta_t(\alpha_{\sigma},\eta^{\sigma}))_{t\in \mathbb{R}_+}\in {\bf \Gamma}^{\odot}(F),$ the product system of time ordered Fock module over $F$ with the $m$-particle $(m > 0)$ sector component $\zeta_t^m$ of $\zeta_t(\alpha_{\sigma},\eta^{\sigma})\in {\bf \Gamma}_t(F)$ given by $$\zeta_t^m(t_m,\dots,t_1)=e^{(t-t_m)\alpha_{\sigma}}\eta^{\sigma}\odot e^{(t_m-t_{m-1})\alpha_{\sigma}}\eta^{\sigma}\odot \dots \odot e^{(t_2-t_1)\alpha_{\sigma}}\eta^{\sigma} e^{t_1\alpha_{\sigma}}$$ and $\zeta_t^0=e^{t\alpha_{\sigma}}$ for $\sigma\in S.$ Then by \cite[Theorem 3]{LS01}, $\zeta^{{\sigma}\odot}(\alpha_{\sigma},\eta^{\sigma})$ is a unit for
		${\bf \Gamma}^{\odot}(F).$ Define a map $\mathfrak{T}_t^{(\alpha,\eta)} :S\times S\to B(M_2(\mathcal{B}))$ by 
		$$\mathfrak{T}_t^{(\alpha_{\sigma},\eta^{\sigma})}(B)=\mathfrak{T}_t^{\sigma,\sigma'}(B)=\langle \zeta_t(\alpha_{\sigma},\eta^{\sigma}),B\zeta_t(\alpha_{\sigma'},\eta^{\sigma'}) \rangle \quad for \quad B\in M_2(\mathcal{B}).$$
		Then the maps $\mathfrak{T}:=(\mathfrak{T}_t)_{t\ge 0}$ is a CPD-semigroup of uniformly continuous maps over $S$ on $M_2(\mathcal{B})$ such that the bounded generator $$\mathfrak{L}^{\sigma,\sigma'}(B)=\mathfrak{L}^{(\alpha,\eta)}(B)=\langle \eta^{\sigma},B\eta^{\sigma'} \rangle +B\alpha_{\sigma'}+\alpha_{\sigma}^*B \quad for \quad B\in M_2(\mathcal{B}).$$ Let $\eta^{\sigma}_i=\mathbb{F}_{ii}\eta^{\sigma}\mathbb{F}_{ii},i\in I_2,$ then $\eta^{\sigma}=\eta^{\sigma}_1+\eta^{\sigma}_2$ and $\langle \eta^{\sigma}_1,\eta^{\sigma}_2\rangle=0.$ Define a map $\tau:S\times S\to B(M_2(\mathcal{B}))$ by $\tau^{\sigma,\sigma'}(B)=\langle \eta^{\sigma},B\eta^{\sigma'} \rangle, B\in M_2(\mathcal{B}),$ then $\tau$ is block CPD-kernel, say $\tau^{\sigma,\sigma'}=\begin{pmatrix}
			\tau_{11}^{\sigma,\sigma'} & \tau_{12}^{\sigma,\sigma'}\\\tau_{12}^{{\sigma,\sigma'}*}& \tau_{22}^{\sigma,\sigma'}
		\end{pmatrix}.$ Note that $(F^{(\mathcal{B})}, [\eta_i])$ is a Kolmogorov-representation for $\tau_{ii}$ over $S,$ $i\in I_2,$ where $F^{(\mathcal{B})}$ is the von Neumann $\mathcal{B}$-$\mathcal{B}$-module (see Remark \ref{obs-new-module}).
		
		Suppose that $F_i = \overline{span}^s \mathcal{B} [\eta_i^{\sigma}]\mathcal{B}\subseteq F^{(\mathcal{B})}$ is the (minimal) Kolmogorov-representation for $\tau_{ii}$ over $S$ $,i\in I_2,$ and suppose
		$V:F_2 \to F_1$ is the unique adjointable bilinear contraction with $\tau^{\sigma,\sigma'}_{12}(a) = \langle[\eta^{\sigma}_1], V a[\eta^{\sigma'}_2]\rangle$  (see Theorem \ref{main-single}). For $\sigma',\sigma\in S$ we have
		\begin{align*}
			\tau^{\sigma,\sigma'}(B)=\langle \eta^{\sigma},B\eta^{\sigma'} \rangle=\begin{pmatrix}
				\langle [\eta^{\sigma}_1],a_{11}[\eta_1^{\sigma'}] \rangle & \langle [\eta^{\sigma}_1],Va_{12}[\eta_2^{\sigma'}] \rangle\\\langle [\eta^{\sigma'}_2],V^*a_{21}[\eta_1^{\sigma}] \rangle& \langle [\eta^{\sigma}_2],a_{22}[\eta_2^{\sigma'}] \rangle
			\end{pmatrix} \quad for \quad B=\begin{pmatrix}
				a_{11} &a_{12}\\a_{21}&a_{22}
			\end{pmatrix}\in M_2(\mathcal{B}).
		\end{align*}Thus 
		\begin{align*}
			\mathfrak{L}^{\sigma,\sigma'}(B)&=\langle \eta^{\sigma},B\eta^{\sigma'} \rangle +B\alpha_{\sigma'}+\alpha_{\sigma}^*B\\&=\begin{pmatrix}
				\langle [\eta^{\sigma}_1],a_{11}[\eta_1^{\sigma'}] \rangle +a_{11}\alpha_{\sigma',1}+\alpha_{\sigma,1}^*a_{11}& \langle [\eta^{\sigma}_1],Va_{12}[\eta_2^{\sigma'}] \rangle +a_{12}\alpha_{\sigma',2}+\alpha_{\sigma,1}^*a_{12}\\\langle [\eta^{\sigma'}_2],V^*a_{21}[\eta_1^{\sigma}] \rangle +a_{21}\alpha_{\sigma',1}+\alpha_{\sigma,2}^*a_{21}& \langle [\eta^{\sigma}_2],a_{22}[\eta_2^{\sigma'}] \rangle  +a_{22}\alpha_{\sigma',2}+\alpha_{\sigma,2}^*a_{22} 
			\end{pmatrix}\\&=\begin{pmatrix}
				\mathfrak{L}_{11}^{(\alpha_1,[\eta_1])}(a_{11}) & \mathfrak{L}_{12}^{(\alpha_1,\alpha_2,[\eta_1],[\eta_2],V)}(a_{12})\\\mathfrak{L}_{21}^{(\alpha_1,\alpha_2,[\eta_1],[\eta_2],V)}(a_{21})& \mathfrak{L}_{22}^{(\alpha_2,[\eta_2])}(a_{22})
			\end{pmatrix},
		\end{align*} where $$\mathfrak{L}_{ii}^{\sigma,\sigma'}(b)=\mathfrak{L}_{ii}^{(\alpha_i,[\eta_i])}(b)=\langle [\eta^{\sigma}_i],b[\eta_i^{\sigma'}] \rangle +b\alpha_{\sigma',i}+\alpha_{\sigma,i}^*b,$$ 
	\begin{equation}\label{Azad11}
		\mathfrak{L}_{12}^{\sigma,\sigma'}(b)=\mathfrak{L}_{12}^{(\alpha_1,\alpha_2,[\eta_1],[\eta_2],V)}(b)=\langle [\eta^{\sigma}_1],Vb[\eta_2^{\sigma'}] \rangle +b\alpha_{\sigma',2}+\alpha_{\sigma,1}^*b \quad \mbox{and}
	\end{equation}
 $\mathfrak{L}_{21}^{\sigma,\sigma'}(b)={\mathfrak{L}_{12}^{\sigma,\sigma'}(b^*)}^*$ for every $ i\in I_2, b\in \mathcal{B}.$ Therefore, for $B=\begin{pmatrix}
			a_{11} &a_{12}\\a_{21}&a_{22}
		\end{pmatrix}\in M_2(\mathcal{B}),$ $$\mathfrak{T}_t^{\sigma,\sigma'}(B)=e^{t\mathfrak{L}^{\sigma,\sigma'}(B)}=\begin{pmatrix}
			e^{t\mathfrak{L}_{11}^{\sigma,\sigma'}(a_{11})} &e^{t\mathfrak{L}_{12}^{\sigma,\sigma'}(a_{12})}\\e^{t\mathfrak{L}_{21}^{\sigma,\sigma'}(a_{21})}&e^{t\mathfrak{L}_{22}^{\sigma,\sigma'}(a_{22})}
		\end{pmatrix}.$$
		Note that $(F^i = (F_t^i), \zeta_i^{\sigma} = (\zeta_{t,i}(\alpha_{\sigma, i},[\eta_i^{\sigma}])))$ is the inclusion system associated to $\mathfrak{K}_{i}=(e^{t\mathfrak{L}_{ii}^{\sigma,\sigma'}})_{t\ge 0}$ is a subsystem of ${\bf \Gamma}^{\odot}(F_i),i\in I_2$ over $F_i.$ Suppose that $\beta = (\beta_t)_{t\ge 0}$ is the contractive morphism from $(F^2,\zeta_2^{\sigma})$ to $(F^1,\zeta_1^{\sigma})$ such that
		\begin{equation}\label{PP1}
			e^{t\mathfrak{L}_{12}^{\sigma,\sigma'}}(b)=\langle \zeta_{t,1}^{\sigma}(\alpha_{\sigma, 1},[\eta_1^{\sigma}]),b \beta_t(\zeta_{t,2}^{\sigma'}(\alpha_{\sigma', 2},[\eta_2^{\sigma'}])) \rangle \quad for \quad b\in \mathcal{B},\sigma',\sigma\in S.
		\end{equation} For morphism maps from a unit to a unit we have
		\begin{equation}\label{PP2}
			\beta_t(\zeta_{t,2}^{\sigma}(\alpha_{\sigma, 2},[\eta_2^{\sigma}]))= \zeta_{t,1}^{\sigma}(\gamma_{\beta}(\alpha_{\sigma, 2},[\eta_2^{\sigma}]),\eta_{\beta}(\alpha_{\sigma, 2},[\eta_2^{\sigma}]))
		\end{equation} for some $\gamma_{\beta}, \eta_{\beta} : \mathcal{B} \times F_2 \to \mathcal{B} \times F_1.$ From Equations (\ref{PP1}) and (\ref{PP2}) we get $$e^{t\mathfrak{L}_{12}^{\sigma,\sigma'}}(b)=\langle \zeta_{t,1}^{\sigma}(\alpha_{\sigma, 1},[\eta_1^{\sigma}]),b \zeta_{t,1}^{\sigma'}(\gamma_{\beta}(\alpha_{\sigma', 2},[\eta_2^{\sigma'}]),\eta_{\beta}(\alpha_{\sigma', 2},[\eta_2^{\sigma'}])) \rangle$$ By differentiating Equation (\ref{PP1}) we obtain
		\begin{equation}\label{Azad1}
			\mathfrak{L}_{12}^{\sigma,\sigma'}(b)=\langle \eta^{\sigma}_1,b \eta_{\beta}(\alpha_{\sigma',2},[\eta_2^{\sigma'}])\rangle + b\gamma_{\beta}(\alpha_{\sigma', 2},[\eta_2^{\sigma'}]) + \alpha^*_{\sigma,1} b.
		\end{equation} From Equations (\ref{Azad11}) and (\ref{Azad1}), we get $\alpha_{\sigma',2}=\gamma_{\beta}(\alpha_{\sigma', 2},[\eta_2^{\sigma'}])$ and $V [\eta_2^{\sigma'}]=\eta_{\beta}(\alpha_{\sigma', 2},[\eta_2^{\sigma'}]),$ and hence  $$\beta_t(\zeta_{t,2}^{\sigma}(\alpha_{\sigma, 2},[\eta_2^{\sigma}])=\zeta_{t,1}^{\sigma}(\alpha_{\sigma,2},V[\eta_2^{\sigma}]).$$
	\end{example}
	
	\subsection{$E_0$-dilation of block quantum Markov semigroups.}

	Now we will discuss if we have a block QMS for completely positive definite kernels over a set $S$ on a given unital $C^*$-algebra, then the construction of $E_0$-dilation in \cite{BBLS04} is a block kernels semigroup. (We change some notations from \cite{BBLS04}: $\check{E_s}\leadsto F_s, %E_\ft\leadsto E_\ft,
	E_s\leadsto \mathcal{F}_s, E\leadsto \mathcal{F}$). 
	
	Suppose that $\mathcal{B}$ is a unital $C^*$-algebra with projections $p\in\mathcal{B}$ and $q=\mathbf{1}-p.$ Suppose that $\mathfrak{T}=(\mathfrak{T}_s)_{s\ge 0}$ is a block QMS for completely positive definite kernels over a set $S$ on $\mathcal{B}$ with respect to $p.$ Assume $(F = (F_s),  \eta^{\sigma\odot} = (\eta_s^{\sigma}))$ to be the inclusion system  associated to $\mathfrak{T}$ over $S,$ then by \cite[Section 4]{BBLS04},
	%Let $(E=(E_t),\xi^\odot=(\xi_t))$ be the inclusion system associated to $\Phi.$ Recall from \cite[Section 4]{BBLS04} that
	\[
	\begin{tikzcd}
		(F_s,\eta_s^{\sigma})\arrow[d, "\text{first inductive limit}"] \\
		(\mathcal{F}_s,\eta^s_{\sigma})\arrow["\text{second inductive limit}", d]\\
		(\mathcal{F},\eta^{\sigma})
	\end{tikzcd}\]
	there exist a $\mathcal{B}$-module $\mathcal{F}$ (using second inductive limit) such that $\mathcal{F} \odot \mathcal{F}_s \simeq \mathcal{F},$
	endomorphisms $\vartheta_s : B^a(\mathcal{F}) \to B^a(\mathcal{F})$ given by $\vartheta_s(b) = b \odot I_{\mathcal{F}_s}$ for $b\in B^a(\mathcal{F})$ with $(\vartheta_s)_{s\ge 0}$ is an $E_0$-dilation
	of $(\mathfrak{T}_s)_{s\ge 0}$ over $S$ and a representation $l_0 : \mathcal{B} \to  B^a(\mathcal{F})$ defined by $l_0(b)= |\eta^{\sigma}\rangle b \langle \eta^{\sigma}|$ for $b\in \mathcal{B}$. Note that the Markov property 
	\begin{equation}\label{PP3}
		l_0(\mathfrak{T}_s^{\sigma,\sigma}(b))=l_0({\bf 1})\vartheta_s({l_0(b)})l_0({\bf 1}) \quad \mbox{for}\quad b\in \mathcal{B},s\ge 0,\sigma\in S.
	\end{equation} This shows that $l_0({\bf 1})\vartheta_s({l_0({\bf 1})})l_0({\bf 1})=l_0(\mathfrak{T}_s^{\sigma,\sigma}({\bf 1}))=l_0({\bf 1}).$ It follows that $l_0({\bf 1})\le \vartheta_s(l_0({\bf 1})).$ Indeed, $l_0({\bf 1})$ is a projection. Therefore $(\vartheta_s(l_0({\bf 1})))_{s\ge 0}$ is a family of projections (increasing), and hence it converges
	in SOT. Suppose that $k_s : \mathcal{F}_s \to \mathcal{F}$ is the (isometries) canonical functions defined by $(\eta^s_{\sigma}\mapsto \eta^{\sigma} \odot \eta^s_{\sigma}),$ then \begin{equation}\label{PP6}
		\overline{span} \:k_s(\mathcal{F}_s)=\mathcal{F}.
	\end{equation} Therefore $\vartheta_s(l_0({\bf 1}))(\mathcal{F}_s) = \vartheta_s(l_0({\bf 1}))(\eta^{\sigma} \odot \mathcal{F}_s) = (|\eta^{\sigma}\rangle \langle \eta^{\sigma}|\odot I_{\mathcal{F}_s} )(\eta^{\sigma} \odot \mathcal{F}_s) = \eta^{\sigma}\odot \mathcal{F}_s,$ and hence $(\vartheta_s(l_0({\bf 1})))_{s\ge 0}$
	is converging in SOT to $I_{\mathcal{F}}.$

	Since $\mathfrak{T}$ is a unital block semigroup, we have $\mathfrak{T}^{\sigma,\sigma}_s(m)=m$ for $m = p$ or $q.$ From Equation (\ref{PP3}), we obtain \begin{equation}\label{PP4}
		l_0({\bf 1})\vartheta_s({l_0(m)})l_0({\bf 1})=l_0(\mathfrak{T}^{\sigma,\sigma}_s(m))=l_0(m) \quad for \quad  m=p,q.
	\end{equation}
	Since $l_0(p) + l_0(q)=l_0({\bf 1})$ and $l_0(q)l_0(p)=0=l_0(p)l_0(q).$ In Equation (\ref{PP4}) multiplying both sides by $l_0(m)$ we have
	$$l_0(m)\vartheta_s({l_0(m)})l_0(m)=l_0(m), \quad for \quad  m=p,q.$$
	Since $l_0(m)$ is a projection, $l_0(m) \le \vartheta_s(l_0(m))$ for $s\ge 0.$ Therefore $\vartheta_s(l_0(m))\le \vartheta_t(l_0(m))$ for
	$s\le t,$ and hence $(\vartheta_s(l_0(m)))_{s\ge 0}\in B^a(\mathcal{F})$ is a family of projections (increasing). Suppose that $(\vartheta_s(l_0(p)))_{s\ge 0}$ converges to $P,$ then 
	$(\vartheta_s(l_0(q)))_{t\ge 0}$ converge to $Q=I_{\mathcal{F}}-P.$
	And hence $PQ = 0,$
	\begin{equation}\label{PP5}
		\overline{span}^s \vartheta_s(l_0(q))(\mathcal{F})=Q(\mathcal{F})	\quad and \quad \overline{span}^s \vartheta_s(l_0(p))(\mathcal{F})=P(\mathcal{F}) .
	\end{equation} Therefore $ \mathcal{F}^{(1)} \oplus \mathcal{F}^{(2)}=\mathcal{F},$ where $\mathcal{F}^{(2)} = Q(\mathcal{F})$ and $\mathcal{F}^{(1)} = P (\mathcal{F}).$
	
	\begin{lemma}\label{PP7}
		$\vartheta_t(l_0(q))(\mathcal{F}_t)=Q(\mathcal{F}_t)$ and $\vartheta_t(l_0(p))(\mathcal{F}_t)=P(\mathcal{F}_t)$ for every $t\ge 0.$
	\end{lemma}
	\begin{proof}
		Fix $t\ge 0,$ let $\eta^{t}_{\sigma}\in \mathcal{F}_t,\sigma\in S$ and $s\ge t,$ it is sufficient to show for $m = q,p,$ that $\vartheta_s(l_0(m))(\eta^{t}_{\sigma}) = \vartheta_t(l_0(m))(\eta^{t}_{\sigma}).$ Since $\langle \eta^t_{\sigma},m\eta^t_{\sigma}\rangle=\mathfrak{T}^{\sigma,\sigma}_t(m)=m$ for $m = p,q$ and $\sigma\in S$ we have 
		\begin{align*}
			\|p\eta^t_{\sigma}- p\eta^t_{\sigma}p \|^2=\| p\eta^t_{\sigma}q \|^2=\| \langle p\eta^t_{\sigma}q,p\eta^t_{\sigma}q\rangle \|=\| q\langle \eta^t_{\sigma},p\eta^t_{\sigma}\rangle q\|=\|q \mathfrak{T}^{\sigma,\sigma'}_t(p)q \|=\|qpq\|=0,
		\end{align*}similarly, $\|\eta^t_{\sigma}p- p\eta^t_{\sigma}p \|^2=0.$ This implies that $\eta_t^{\sigma}p=p\eta_t^{\sigma}=p\eta_t^{\sigma}p.$ On the other side $\eta^t_{\sigma}q=q\eta^t_{\sigma}=q\eta^t_{\sigma}q.$ Let $y\in \mathcal{F}_t, m=p,q$ and $s\ge t,$ then $\eta^{s-t}_{\sigma}\odot y\in \mathcal{F}_s$ and 
		\begin{align*}
			\vartheta_s(l_0(m))(y)&=(|\eta^{\sigma}\rangle m\langle \eta^{\sigma}|\odot I_{\mathcal{F}_s} )(\eta^{\sigma} \odot \eta^{s-t}_{\sigma}\odot y)=\eta^{\sigma} \odot m\eta^{s-t}_{\sigma}\odot y=\eta^{\sigma} \odot \eta^{s-t}_{\sigma}m \odot y\\&=\eta^{\sigma} \odot \eta^{s-t}_{\sigma}\odot my=\eta^{\sigma} \odot m y=\vartheta_t(l_0(m))(y).
		\end{align*}
	\end{proof}
	From \cite[Theorem 4.4.3]{BBLS04}, $\mathcal{F} \odot \mathcal{F}_s \simeq \mathcal{F}$ for $s \ge 0.$
	Now we present a similar result for $\mathcal{F}^{(i)}$’s.
	\begin{lemma}
		$\mathcal{F}^{(i)} \odot \mathcal{F}_s \simeq \mathcal{F}^{(i)}$ for every $i\in I_2$ and $s \ge 0.$
	\end{lemma}
	\begin{proof}
		Suppose that $k_s : \mathcal{F}_s \to \mathcal{F}$ is the (isometries) canonical maps. Define a unitary $v_s : \mathcal{F}\odot \mathcal{F}_s \to \mathcal{F}$ (see \cite[Theorem 4.4.3]{BBLS04}) by $$v_s(k_t(y_t)\odot x_s)=k_{t+s}(y_t\odot x_s) \quad \mbox{for all}\quad y_t\in \mathcal{F}_t,x_s\in \mathcal{F}_s,$$ then $\mathcal{F} \odot \mathcal{F}_s \simeq \mathcal{F}.$ Since $ \mathcal{F}^{(1)} \oplus \mathcal{F}^{(2)}=\mathcal{F}$ we get $ \mathcal{F}^{(1)} \odot \mathcal{F}_s \oplus \mathcal{F}^{(2)} \odot \mathcal{F}_s \simeq \mathcal{F}.$ Now, we want to show the restriction map of the unitary $v_s|_{\mathcal{F}^{(i)} \odot \mathcal{F}_s}:\mathcal{F}^{(i)} \odot \mathcal{F}_s\to \mathcal{F}^{(i)}$ is a unitary. It is sufficient to show $v_s(\mathcal{F}^{(i)} \odot \mathcal{F}_s) \subseteq \mathcal{F}^{(i)}.$ From Equations (\ref{PP6}), (\ref{PP5}) and Lemma \ref{PP7}, it is enough to show $v_s(\vartheta_t(l_0(q))k_t(\mathcal{F}_t)\odot \mathcal{F}_s) \subseteq \mathcal{F}^{(2)}$ and $v_s(\vartheta_t(l_0(p))k_t(\mathcal{F}_t)\odot \mathcal{F}_s) \subseteq \mathcal{F}^{(1)}.$
		Let $y_t\in \mathcal{F}_t,x_s\in \mathcal{F}_s$ and $m = p$ or $q$ we have
		\begin{align*}
			&v_s(\vartheta_t(l_0(m))k_t(y_t)\odot x_s)=v_s((\eta^{\sigma} m\odot y_t)\odot x_s)=v_s((\eta^{\sigma} \odot my_t)\odot x_s)=v_s(k_t(m y_t)\odot x_s)\\&=k_{t+s}(m y_t\odot x_s)=\eta^{\sigma}\odot my_t\odot x_s=\eta^{\sigma}m\odot y_t\odot x_s=\vartheta_{t+s}(l_0(m))k_{t+s}(y_t\odot x_s),
		\end{align*} and hence $v_s(\vartheta_t(l_0(m))k_t(y_t)\odot x_s)\in \mathcal{F}^{(1)}$ for $m= p$ and $v_s(\vartheta_t(l_0(m))k_t(y_t)\odot x_s)\in\mathcal{F}^{(2)}$ for $m=q.$ 
	\end{proof}
	
	\begin{theorem}
		The $E_0$-dilation $\vartheta=(\vartheta_s)_{s\ge 0}$
		of $\mathfrak{T}$ over $S$ defined as above is a block kernels semigroup  with respect to projection $P.$
	\end{theorem}
	\begin{proof}
		Since $\mathcal{F}^{(1)} \oplus \mathcal{F}^{(2)}=\mathcal{F}$ we get
		$$B^a(\mathcal{F})=\begin{pmatrix}
			B^a(\mathcal{F}^{(1)}) &B^a(\mathcal{F}^{(2)},\mathcal{F}^{(1)})\\B^a(\mathcal{F}^{(1)},\mathcal{F}^{(2)})&B^a(\mathcal{F}^{(2)})
		\end{pmatrix}.$$
		Let $b\in B^a(\mathcal{F}^{(i)},\mathcal{F}^{(j)})$ and $i,j\in I_2$ we have 
		\begin{align*}
			\vartheta_s(b)=b\odot I_{\mathcal{F}_s}\in B^a(\mathcal{F}^{(i)}\odot \mathcal{F}_s,\mathcal{F}^{(j)}\odot \mathcal{F}_s)=B^a(\mathcal{F}^{(i)},\mathcal{F}^{(j)}),
		\end{align*} and hence $\vartheta_s$ acts block-wise.
	\end{proof}
	
	\section{Lifting of morphisms}
	Let $\mathfrak{L}=(\mathfrak{L}_t)_{t\in \mathbb{T}}$ and $\mathfrak{K}=(\mathfrak{K}_t)_{t\in \mathbb{T}}$ be two normal QDS over $S$ on a vNa $\mathcal{B}.$ Suppose that $(F, \alpha)$ and $(E, \beta)$  are two inclusion systems corresponding $\mathfrak{L}=(\mathfrak{L}_t)_{t\in \mathbb{T}}$ and $\mathfrak{K}=(\mathfrak{K}_t)_{t\in \mathbb{T}},$ respectively.  We establish in this section that any morphism between $(F, \alpha)$ and $(E, \beta)$ can be lifted as a unique morphism between the product systems generated by $(F, \alpha)$ and $(E, \beta),$ respectively. We will use the following results and notations from \cite{BBLS04} and \cite{BK20}. For $s\ge 0,$ define \begin{equation*}
		\mathbb{J}_s:=\{{\mathfrak s}=(s_n,\dots,s_1)\in \mathbb{T}^n: s_i\ge 0, |{\mathfrak s}|=s, n\in \mathbb{N}\}.
	\end{equation*} Let ${\mathfrak t}=(t_m,\dots,t_1)\in \mathbb{J}_t$ and ${\mathfrak s}=(s_n,\dots,s_1)\in \mathbb{J}_s$ the {\it joint tuple} ${\mathfrak t}\smile{\mathfrak s}\in \mathbb{J}_{t+s}$ defined as $${\mathfrak t}\smile{\mathfrak s}:=((t_m,\dots,t_1),(s_n,\dots,s_1))=(t_m,\dots,t_1,s_n,\dots,s_1).$$
	We equip $\mathbb{J}_s$ with a {\it partial order} by saying ${\mathfrak s}\ge {\mathfrak t} = (t_m,\dots,t_1),$ if for every $j \:(1 \le j \le m)$
	there exist ${\mathfrak t}_j\in \mathbb{J}_{t_j}$ such that ${\mathfrak s}={\mathfrak t}_m\smile \dots \smile {\mathfrak t}_1.$ By $()$ we denote the {empty tuple}, for $s = 0,$  $\mathbb{J}_0 = \{()\}.$ 
	
	Using the inductive limits we will now discuss the construction of a product system generated by a normal QDS from \cite{BBLS04}. Let $\mathfrak{L}=(\mathfrak{L}_t)_{t\in \mathbb{T}}$ be a normal QDS over a set $S$ on a vNa $\mathcal{B}$ with the Kolmogorov-representation $(F_t,\eta_t)$ for $\mathfrak{L}_t.$ Then by Remark \ref{DS3},  $(F = (F_t), \alpha = (\alpha_{t,s}))$ is the inclusion system and $\eta^{\sigma \odot }=(\eta_t^{\sigma})$ is the generating unit. For fix $t\in \mathbb{T},$ suppose that $F_{\mathfrak t}=F_{t_m}\odot\dots\odot F_{t_1}$ for every ${\mathfrak t}=(t_m,\dots,t_1)\in \mathbb{J}_t,$
	then $\alpha_{{\mathfrak t}(t)}:F_t\to F_{\mathfrak t}$ defined as $$\alpha_{{\mathfrak t}(t)}=(\alpha_{t_m,t_{m-1}}\odot I)(\alpha_{t_m+t_{m-1},t_{m-2}}\odot I)\dots (\alpha_{t_m+\dots t_{3},t_2}\odot I)\alpha_{t_m+\dots t_{2},t_1}.$$ Let ${\mathfrak t}=(t_m,\dots,t_1)={\mathfrak s}_n\smile \dots \smile {\mathfrak s}_1\ge {\mathfrak s}=(s_n,\dots,s_1)$ with $|{\mathfrak s}_j|=s_j,$ define bilinear isometries $\alpha_{{\mathfrak t}{\mathfrak s}}:F_{\mathfrak s}\to F_{\mathfrak t}$ by $$\alpha_{{\mathfrak t}{\mathfrak s}}:=\alpha_{{\mathfrak s}_n(s_n)}\odot\dots \odot \alpha_{{\mathfrak s}_1(s_1)},$$ then $\alpha_{{\mathfrak t}{\mathfrak s}}\alpha_{{\mathfrak s}{\mathfrak r}}=\alpha_{{\mathfrak t}{\mathfrak r}}$ for all
	${\mathfrak t}\ge {\mathfrak s}\ge {\mathfrak r}.$ It gives the family $(F_{\mathfrak t})_{{\mathfrak t}\in \mathbb{J}_t}$ is an {inductive system}, and hence the inductive limit $\mathcal{F}_t=\lim_{{\mathfrak t}\in \mathbb{J}_t}ind F_{\mathfrak t}$ is the von Neumann $\mathcal{B}$-$\mathcal{B}$-module with the
	(bilinear isometries) canonical mappings $j_{\mathfrak t}:F_{\mathfrak t}\to \mathcal{F}_{t}$ (see \cite[Proposition 4.3.1]{BBLS04}). Let  ${\mathfrak s}\in \mathbb{J}_s$ and ${\mathfrak t}\in \mathbb{J}_t,$ clearly $F_{{\mathfrak s}\smile {\mathfrak t}}=F_{\mathfrak s}\odot F_{{\mathfrak t}}.$ We define $C_{s,t}:\mathcal{F}_{s+t}\to \mathcal{F}_s\odot \mathcal{F}_t$ by
	$$C_{s,t}j_{{\mathfrak s}\smile {\mathfrak t}}(y_{\mathfrak s}\odot x_{\mathfrak t})=j_{\mathfrak s}y_{\mathfrak s}\odot j_{\mathfrak t}x_{\mathfrak t},$$ where $y_{\mathfrak s}\in F_{\mathfrak s},x_{\mathfrak t}\in F_{\mathfrak t},{\mathfrak t}\in \mathbb{J}_t$ and ${\mathfrak s}\in \mathbb{J}_s.$ Then by \cite[Theorem 4.3.5]{BBLS04}, $(\mathcal{F}=(\mathcal{F}_t)_{t\in \mathbb{T}},C=(C_{s,t})_{s,t \in \mathbb{T}})$ is a product system. We say that the product system $(\mathcal{F}, C)$ is {generated by the inclusion system} $(F, \alpha).$
	
	Now we recall the following: Suppose that $\mathcal{B}$ is a vNa on $\mathcal{H}.$ Let $\mathfrak{L}$ be a CPD-kernel over a set $S$ on $\mathcal{B},$  and let $(F,\eta)$ denote the Kolmogorov-representation for $\mathfrak{L}.$ Let ${K}=F\odot \mathcal{H}$ be a Hilbert space. For $\sigma \in S,$ defined a bounded map $L_{\eta^{\sigma}}:\mathcal{H}\to K$ by $L_{\eta^{\sigma}}(g)=\eta^{\sigma} \odot g$ for all $g\in \mathcal{H},$ consider $F$ as a concrete subset of $B(\mathcal{H},K).$ Sometimes we write $\eta^{\sigma}g$ instead of $\eta^{\sigma}\odot g.$
	\begin{remark}\label{DJ3}
		Suppose that $\mathcal{B}$ is a vNa on $\mathcal{H}$ 
		and  $(\mathcal{F}, C)$ is the product system defined above. Suppose that $j_{\mathfrak t}:F_{\mathfrak t}\to \mathcal{F}_{t}$ is the canonical mappings for ${\mathfrak t}\in \mathbb{J}_t,$ then
		\begin{equation}\label{DJ5}
			\lim_{{\mathfrak t}\in \mathbb{J}_t}\|y h-j_{\mathfrak t}j_{\mathfrak t}^*(y)h\|=0 \quad \mbox{for every}\quad y\in \mathcal{F}_t,h\in \mathcal{H}.
		\end{equation} That is, $j_{\mathfrak t}j_{\mathfrak t}^*$ increases to identity in SOT.
		
	\end{remark}
	
	Now we present a lifting theorem almost same proof
	of \cite[Theorem 5.3]{BK20} and \cite[Theorem 11]{BM10}.
	\begin{theorem}
		Suppose that $\mathcal{B}$ is a vNa on $\mathcal{H}.$ Let $\mathfrak{L}=(\mathfrak{L}_t)_{t\in \mathbb{T}}$ and $\mathfrak{K}=(\mathfrak{K}_t)_{t\in \mathbb{T}}$ be two normal QDS over a set $S$ on $\mathcal{B}.$ Suppose that $(F, \alpha)$ and $(E, \beta)$  are
		two inclusion systems corresponding $\mathfrak{L}=(\mathfrak{L}_t)_{t\in \mathbb{T}}$ and $\mathfrak{K}=(\mathfrak{K}_t)_{t\in \mathbb{T}}$ generating two product systems $(\mathcal{F}, C)$ and $(\mathcal{E}, B),$ respectively, defined as above. Suppose $j$ and $i$ are the respective inclusion maps and $V:(E, \beta) \to  (F, \alpha)$ is a morphism, then there is a unique morphism $\widehat{V}:(\mathcal{E}, B)\to (\mathcal{F}, C)$ with $$V_s=j_s^*\widehat{V}_si_s \quad \mbox{for}\quad s\in \mathbb{T}.$$
	\end{theorem}
	\begin{proof}
		Suppose $V:(E, \beta) \to  (F, \alpha)$ is a morphism, then for some $k$ we have $\|V_s\|\le e^{ks}$ for $s\in \mathbb{T}.$ Let ${\mathfrak s} = (s_n, ..., s_1)\in \mathbb{J}_s,$ we define $V_{\mathfrak s}:E_{\mathfrak s}\to F_{\mathfrak s}$ by $V_{\mathfrak s}=V_{s_n}\odot\dots\odot V_{s_1}.$ Suppose that $j_{\mathfrak s} : F_{\mathfrak s}\to \mathcal{F}_s$ and $i_{\mathfrak s} : E_{\mathfrak s}\to \mathcal{E}_s$ are the canonical maps (bilinear isometries). Let ${\mathfrak s}\le {\mathfrak t}$ in $\mathbb{J}_s,$ then \begin{equation}\label{DJ4}
			\alpha^*_{{\mathfrak t}{\mathfrak s}}V_{\mathfrak t}\beta_{{\mathfrak t}{\mathfrak s}}=V_{\mathfrak s}.
		\end{equation}
		Let ${\mathfrak s}\in \mathbb{J}_s,$ define $\mathfrak{T}_{{\mathfrak s}} := j_{\mathfrak s}V_{\mathfrak s}i^*_{\mathfrak s}.$ Let $Q_{\mathfrak s}=i_{\mathfrak s}i^*_{\mathfrak s}$ and $P_{\mathfrak s}=j_{\mathfrak s}j^*_{\mathfrak s}.$ From Remark \ref{DJ3}, $(Q_{{\mathfrak s}})_{{\mathfrak s}\in \mathbb{J}_s}$ and $(P_{{\mathfrak s}})_{\mathfrak s \in \mathbb{J}_s}$
		are increasing projections families. Let ${\mathfrak r}\le {\mathfrak s},$ then $i_{\mathfrak s}\beta_{{\mathfrak s}{\mathfrak r}}=i_{{\mathfrak r}}$ and $j_{\mathfrak s}\alpha_{{\mathfrak s}{\mathfrak r}}=j_{{\mathfrak r}},$ and hence $i_{\mathfrak s}^*i_{\mathfrak r}=\beta_{{\mathfrak s}{\mathfrak r}}$ and $j_{\mathfrak s}^*j_{\mathfrak r}=\alpha_{{\mathfrak s}{\mathfrak r}}.$ From Equation (\ref{DJ4}), we get $\mathfrak{T}_{{\mathfrak r}}=P_{\mathfrak r}\mathfrak{T}_{{\mathfrak s}}Q_{\mathfrak r}.$
		
		For $s\in \mathbb{T},$ $\mathcal{F}_s \subseteq B(\mathcal{H}, \mathcal{F}_s\odot \mathcal{H})$ and $\mathcal{E}_s \subseteq B(\mathcal{H}, \mathcal{E}_s\odot \mathcal{H}).$ Let $s\in \mathbb{T}$ fix, $y\in \mathcal{E}_s,h\in \mathcal{H},\delta> 0$ and from Equation (\ref{DJ5}) choose ${{\mathfrak r}_0}\in \mathbb{J}_{s}$ such that
		\begin{equation}\label{DJ6}
			e^{ks}\|Q_{{\mathfrak r}_0}(y)h-yh\|< \frac{\delta}{3}.
		\end{equation}This implies that
		\begin{align*}
			\|\mathfrak{T}_{{\mathfrak s}}(y)h-\mathfrak{T}_{{\mathfrak s}}Q_{{\mathfrak r}_0}(y)h\|&=\|\mathfrak{T}_{{\mathfrak s}}(y)\odot h-\mathfrak{T}_{{\mathfrak s}}Q_{{\mathfrak r}_0}(y)\odot h\|=\|(\mathfrak{T}_{{\mathfrak s}}\odot I_{\mathcal{H}})(y\odot h-Q_{{\mathfrak r}_0}(y)\odot h)\|\\&\le \|\mathfrak{T}_{{\mathfrak s}}\odot I_{\mathcal{H}}\|\|y h-Q_{{\mathfrak r}_0}(y)h\|\le e^{ks}\|y h-Q_{{\mathfrak r}_0}(y)h\|< \frac{\delta}{3}.
		\end{align*}
		For ${\mathfrak t}\ge {\mathfrak s}\ge {\mathfrak r_0}\in \mathbb{J}_s,$ since $(Q_{{\mathfrak s}})_{{\mathfrak s}\in \mathbb{J}_s}$ and $(P_{{\mathfrak s}})_{\mathfrak s \in \mathbb{J}_s}$ are families of increasing projections, we get 
		\begin{align}\label{DJ7}
			\nonumber	\|\mathfrak{T}_{{\mathfrak t}}Q_{{\mathfrak r}_0}(y)h\|^2&=\|P_{{\mathfrak t}}\mathfrak{T}_{{\mathfrak t}}Q_{{\mathfrak r}_0}(y)h\|^2=\|P_{{\mathfrak s}}\mathfrak{T}_{{\mathfrak t}}Q_{{\mathfrak r}_0}(y)h-(P_{{\mathfrak t}}-P_{{\mathfrak s}})\mathfrak{T}_{{\mathfrak t}}Q_{{\mathfrak r}_0}(y)h\|^2\\&
			\nonumber=\|P_{{\mathfrak s}}\mathfrak{T}_{{\mathfrak t}}Q_{{\mathfrak r}_0}(y)h\|^2+\|(P_{{\mathfrak t}}-P_{{\mathfrak s}})\mathfrak{T}_{{\mathfrak t}}Q_{{\mathfrak r}_0}(y)h\|^2\\&
			\nonumber=\|P_{{\mathfrak s}}\mathfrak{T}_{{\mathfrak t}}Q_{{\mathfrak s}}Q_{{\mathfrak r}_0}(y)h\|^2+\|\mathfrak{T}_{{\mathfrak t}}Q_{{\mathfrak r}_0}(y)h-P_{{\mathfrak s}}\mathfrak{T}_{{\mathfrak t}}Q_{{\mathfrak s}}Q_{{\mathfrak r}_0}(y)h\|^2\\&=\|\mathfrak{T}_{{\mathfrak s}}Q_{{\mathfrak r}_0}(y)h\|^2+\|\mathfrak{T}_{{\mathfrak t}}Q_{{\mathfrak r}_0}(y)h-\mathfrak{T}_{{\mathfrak s}}Q_{{\mathfrak r}_0}(y)h\|^2.
		\end{align} Therefore $\|\mathfrak{T}_{{\mathfrak t}}Q_{{\mathfrak r}_0}(y)h\|^2\ge \|\mathfrak{T}_{{\mathfrak s}}Q_{{\mathfrak r}_0}(y)h\|^2$ for ${\mathfrak t}\ge {\mathfrak s}\ge {\mathfrak r_0}\in \mathbb{J}_s,$ and hence $$\|\mathfrak{T}_{{\mathfrak s}}Q_{{\mathfrak r}_0}(y)h\|^2\le \|\mathfrak{T}_{{\mathfrak t}}Q_{{\mathfrak r}_0}\|^2\|y\|^2\|h\|^2\le e^{2ks}\|y\|^2\|h\|^2\quad for \quad {\mathfrak s}\in \mathbb{J}_s.$$ It follows that $(\|\mathfrak{T}_{{\mathfrak s}}Q_{{\mathfrak r}_0}(y)h\|^2)_{{\mathfrak s}\in \mathbb{J}_s}$ is Cauchy net, and hence for ${\mathfrak r_1}\in \mathbb{J}_s$ such that ${\mathfrak r_1}\ge {\mathfrak r_0}$  and
		\begin{equation}\label{DJ8}
			|\|\mathfrak{T}_{{\mathfrak t}}Q_{{\mathfrak r}_0}(y)h\|^2-\|\mathfrak{T}_{{\mathfrak s}}Q_{{\mathfrak r}_0}(y)h\|^2| < \frac{\delta^2}{9}\quad \mbox{for all} \quad {\mathfrak t}\ge {\mathfrak s}\ge {\mathfrak r_1}\ge {\mathfrak r_0}\in \mathbb{J}_s.
		\end{equation} From Equations (\ref{DJ7}) and (\ref{DJ8}) and for all ${\mathfrak t}\ge {\mathfrak s}\ge {\mathfrak r_1}\in \mathbb{J}_s,$ we get
		\begin{align*}
			\|(\mathfrak{T}_{{\mathfrak t}}-\mathfrak{T}_{{\mathfrak s}})(y)h\|&\le \|\mathfrak{T}_{{\mathfrak t}}(y)h-\mathfrak{T}_{{\mathfrak t}}Q_{{\mathfrak r}_0}(y)h\|+\|\mathfrak{T}_{{\mathfrak t}}Q_{{\mathfrak r}_0}(y)h-\mathfrak{T}_{{\mathfrak s}}Q_{{\mathfrak r}_0}(y)h\|+\|\mathfrak{T}_{{\mathfrak s}}Q_{{\mathfrak r}_0}(y)h-\mathfrak{T}_{{\mathfrak s}}(y)h\| < \delta.
		\end{align*} Thus $\lim_{{\mathfrak s}\in \mathbb{J}_s}\mathfrak{T}_{{\mathfrak s}}(y)h$
		exists. Let $s\ge 0,$ we define a (bilinear) bounded map $\widehat{V}_s:\mathcal{E}_s\to \mathcal{F}_s$ by $$\widehat{V}_s(y)h:=\lim_{{\mathfrak s}\in \mathbb{J}_s}\mathfrak{T}_{{\mathfrak s}}(y)h.$$ Let $y_{\mathfrak s}\in E_{\mathfrak s},{\mathfrak s}\in \mathbb{J}_s$ and $h\in\mathcal{H},$ we get
		\begin{align*}
			j_{\mathfrak s}^*\widehat{V}_si_{\mathfrak s}(y_{\mathfrak s})h&=\lim_{{\mathfrak r}\in \mathbb{J}_s}j_{\mathfrak s}^*\mathfrak{T}_{{\mathfrak r}}i_{\mathfrak s}(y_{\mathfrak s})h=\lim_{{\mathfrak r}\in \mathbb{J}_s}j_{\mathfrak s}^*j_{\mathfrak r}V_{\mathfrak r}i^*_{\mathfrak r}i_{\mathfrak s}(y_{\mathfrak s})h=\lim_{{\mathfrak r}\in \mathbb{J}_s}\alpha^*_{\mathfrak r{\mathfrak s}}V_{\mathfrak r}\beta_{\mathfrak r{\mathfrak s}}(y_{\mathfrak s})h=V_{\mathfrak s}(y_{\mathfrak s})h.
		\end{align*} This shows that $j_{\mathfrak s}^*\widehat{V}_si_{\mathfrak s}=V_{\mathfrak s}$ for  $s\in\mathbb{T}$ and ${\mathfrak s}\in \mathbb{J}_s.$ In particular $j_{s}^*\widehat{V}_si_{s}=V_{s}$ for $s\in\mathbb{T}.$
		It is sufficient to show that $(\widehat{V}_t)_{t\in \mathbb{T}}$ is a morphism. Let ${\mathfrak s}\in\mathbb{J}_s,{\mathfrak t}\in\mathbb{J}_t$ and $z_{\mathfrak t}\in F_{\mathfrak t},z_{\mathfrak s}\in F_{\mathfrak s},y_{\mathfrak t}\in E_{\mathfrak t},y_{\mathfrak s}\in E_{\mathfrak s}$ we have
		\begin{align*}
			&\langle C^*_{s,t}(\widehat{V}_s\odot \widehat{V}_t)B_{s,t}i_{{\mathfrak s}\smile {\mathfrak t}}(y_{\mathfrak s}\odot y_{\mathfrak t}),j_{{\mathfrak s}\smile {\mathfrak t}}(z_{\mathfrak s}\odot z_{\mathfrak t})\rangle=\langle \widehat{V}_si_{\mathfrak s}y_{\mathfrak s}\odot \widehat{V}_ti_{\mathfrak t}y_{\mathfrak t},j_{\mathfrak s}z_{\mathfrak s}\odot  j_{\mathfrak t} z_{\mathfrak t}\rangle\\&=\langle j_{\mathfrak t}^*\widehat{V}_ti_{\mathfrak t}y_{\mathfrak t},\langle j_{\mathfrak s}^* \widehat{V}_si_{\mathfrak s}y_{\mathfrak s},z_{\mathfrak s}\rangle   z_{\mathfrak t}\rangle= \langle (V_{\mathfrak s}\odot V_{\mathfrak t})(y_{\mathfrak s}\odot y_{\mathfrak t}), z_{\mathfrak s}\odot  z_{\mathfrak t}\rangle=\langle V_{{\mathfrak s}\smile{\mathfrak t}}(y_{\mathfrak s}\odot y_{\mathfrak t}), z_{\mathfrak s}\odot  z_{\mathfrak t}\rangle\\&=\langle \widehat{V}_{s+t}i_{{\mathfrak s}\smile{\mathfrak t}}(y_{\mathfrak s}\odot y_{\mathfrak t}),j_{{\mathfrak s}\smile{\mathfrak t}} (z_{\mathfrak s}\odot  z_{\mathfrak t})\rangle.
		\end{align*} Therefore $C^*_{s,t}(\widehat{V}_s\odot \widehat{V}_t)B_{s,t}=\widehat{V}_{s+t}$ for $s,t\in \mathbb{T}.$
	\end{proof}

	\section{Factorization theorem of a semigroup of $\mathfrak{K}$-families}
In this section, we will discuss the concept of a semigroup of $\mathfrak{K}$-families on a Hilbert $C^*$-module over a $C^*$-algebra. 
A family of CPD-kernels $\{\mathfrak{K}_s\}_{s\in \mathbb T}$ over a set $S$ on a $C^*$-algebra $\mathcal{B}$ forms a {\it CPD-semigroup} or {\it semigroup of CPD-kernels}
if $\mathfrak{K}_s\circ\mathfrak{K}_r=\mathfrak{K}_{s+r}$ for all $s,r\in \mathbb{T}=\mathbb{N}_0$ or $\mathbb{R}_{+}$. 

\begin{definition}
	Let $S$ be a set and $\mathcal{B}$ be a $C^*$-algebra. Suppose that $E$ is a Hilbert $C^*$-module over $\mathcal{B}$ and $\mathcal{K}^{\sigma}_s:E\to E$  is a map 
	for all $s\in \mathbb{T}$ and $\sigma\in S.$ A semigroup $\{\{\mathcal{K}^{\sigma}_s\}_{\sigma\in S}:s\in \mathbb{T}\}$ is said to be
	\begin{enumerate}[(a)]
		\item a {\rm CPD-semigroup} on $E$ if $\{\{\mathcal{K}^{\sigma}_s\}_{\sigma\in S}:s\in \mathbb{T}\}$ extends to block-wise  semigroup of CPD-kernels $\begin{pmatrix}
			{\mathfrak{K}_s^{\sigma,\sigma'}} &  {\mathcal{K}_s^{\sigma^*}} \\
			\mathcal{K}_s^{\sigma'} & \vartheta_s
		\end{pmatrix}$
		on the linking algebra of $E$.
		\item a {\rm CPDH-semigroup} on $E$ if $\{\{\mathcal{K}^{\sigma}_s\}_{\sigma\in S}:s\in \mathbb{T}\}$ is a CPD-semigroup, where $\vartheta_s$ is an $E$-semigroup and for all $s\in \mathbb{T},$ choose the kernel $\{\mathfrak K_s^{\sigma,\sigma'}:\sigma',\sigma\in S\}$ such that $\{\mathcal{K}^\sigma_s\}_{\sigma\in S}$ is a $\mathfrak K_s$-family.
	\end{enumerate}
	
\end{definition}

Now we present a semigroup version of Theorem \ref{DS5}.
\begin{theorem}\label{main result}
	Suppose that ${E}$ is a full Hilbert $C^*$-module over a unital $C^*$-algebra $\mathcal{B}.$ Let $S$ be a set, $\mathcal{K}_t^{\sigma}$ be a family of linear functions on $E$ for every $t\in \mathbb{T},\sigma\in S$ and denote $\mathcal{E}_t:=\overline{span}\{\mathcal{K}^{\sigma}_t(x)b: b\in \mathcal{B},x\in E,\sigma\in S, t\in \mathbb{T}\}.$ Then the following conditions are equivalent:
	\begin{enumerate}
		\item [(a)] There is a semigroup $\mathfrak{K}=(\mathfrak{K}_t)_{t\in \mathbb{T}}$ of CPD-kernels over $S$ on $\mathcal{B}$ such that a semigroup $\mathcal{K}=\{\{\mathcal{K}^{\sigma}_t\}_{\sigma\in S}:t\in \mathbb{T}\}$ is a $\mathfrak{K}_t$-family.
		\item [(b)] $\{\{\mathcal{K}^{\sigma}_t\}_{\sigma\in S}:t\in \mathbb{T}\}$ extends to block-wise semigroup of CPD-kernels $\begin{pmatrix}
			{\mathfrak{K}_t^{\sigma,\sigma'}} &  {\mathcal{K}_t^{\sigma^*}} \\
			\mathcal{K}_t^{\sigma'} & \vartheta_{t}
		\end{pmatrix}:\begin{pmatrix}
			\mathcal{B} &  E^* \\
			E & B^a(E)
		\end{pmatrix}\to \begin{pmatrix}
			\mathcal{B} &  \mathcal{E}_t^* \\
			\mathcal{E}_t & B^a(\mathcal{E}_t)
		\end{pmatrix},$ where  $(\vartheta_{t})_{t\in \mathbb{T}}$ is a semigroup of $*$-homomorphisms.
		\item [(c)] For $\sigma_1,\ldots,\sigma_n\in S$ the family of maps from $E_n$ to $E_n$ given by
		$${\bf x}\mapsto({\mathcal{K}_t^{\sigma_1}}(x_1),{\mathcal{K}_t^{\sigma_2}}(x_2),\ldots,{\mathcal{K}_t^{\sigma_n}}(x_n))~\mbox{for}~{\bf x}=(x_1,x_2\ldots,x_n)\in E_n,t\in \mathbb{T}$$ is a semigroup of completely bounded maps and $\mathcal{E}_t$ can be made into
		a $B^a ({E})$-$\mathcal{B}$-correspondences with 
		$\mathcal{K}_t^{\sigma}$ is a left $ B^a
		({E})$-linear map for all $\sigma\in S,t\in \mathbb{T}.$

		\item [(d)] For $\sigma_1,\ldots,\sigma_n\in S$ the family of maps from $E_n$ to $E_n$ given by
		$${\bf x}\mapsto({\mathcal{K}_t^{\sigma_1}}(x_1),{\mathcal{K}_t^{\sigma_2}}(x_2),\ldots,{\mathcal{K}_t^{\sigma_n}}(x_n))~\mbox{for}~{\bf x}=(x_1,x_2\ldots,x_n)\in E_n,t\in \mathbb{T}$$ is a semigroup of completely bounded maps and $\{\mathcal{K}_t^{\sigma}\}_{\sigma\in S}$
		satisfies
		$$\langle \mathcal{K}_t^\sigma{( y)},\mathcal{K}_t^{\sigma'}({ z}\langle { z}',{ y}'\rangle)\rangle=\langle \mathcal{K}_t^{\sigma}({ z}'\langle {z},{ y}\rangle),\mathcal{K}_t^{\sigma'}({y}')\rangle~\mbox{for all}\quad t\in \mathbb{T},{z},{ z}',y,{ y}'\in{E},\sigma',\sigma\in S.$$ 
	\end{enumerate}
\end{theorem}
\begin{proof}
	$(a)\Rightarrow (b):$
	Let $\mathfrak{K}=(\mathfrak{K}_t)_{t\in \mathbb{T}}$ be a semigroup of CPD-kernels over $S$ on $\mathcal{B}$ with  the (minimal) Kolmogorov-representation $(F_t, \eta_t)$ for $\mathfrak{K}_t.$ Let $(F_s, \eta_s^{\sigma})$ and $(F_{s+t}, \eta_{s+t}^{\sigma})$ be the
	Kolmogorov-representations for $\mathfrak{K}_s$ and $\mathfrak{K}_{s+t},$ respectively. From Equation (\ref{DPS1}), $(F_t\odot F_s ,\eta_{t}^{\sigma}\odot \eta_{s}^{\sigma})$ is a Kolmogorov-representation for $\mathfrak{K}_{s+t}.$
	%	Thus the mapping
	%	\begin{equation}
		%		\mathfrak{k}(\sigma)\mapsto \mathfrak{i}(\sigma)\odot \mathfrak{j}(\sigma),
		%	\end{equation}
	%	extends as a unique bilinear isometry $\mu$ from $K$ to $E\odot F.$ Hence we may identify $K$ as the submodule  $\overline{span}(\A\mathfrak{i}(\sigma)\odot\mathfrak{j}(\sigma)\CC)$ of $E\odot F.$
	%	Note that $E\odot F=\overline{span}(\A\mathfrak{i}(\sigma)\B\odot\B\mathfrak{j}(\sigma)\CC)=\overline{span}(\A\mathfrak{i}(\sigma)\odot\B\mathfrak{j}(\sigma)\CC)=\overline{span}(\A\mathfrak{i}(\sigma)\B\odot\mathfrak{j}(\sigma)\CC).$
	Define bilinear isometries $\alpha_{s,t}:F_{s+t}\to F_s \odot F_t$ by
	\begin{equation*}
		\eta^{\sigma}_{t+s}\mapsto \eta_s^{\sigma}\odot \eta_t^{\sigma}.
	\end{equation*} So from Remark \ref{DS3}, $(F = (F_t), \alpha = (\alpha_{s,t}))$ is an inclusion system. Consider $F_{s+t}$ as the submodule $\overline{span}(\mathcal{B}\eta_{t}^{\sigma}\odot \eta_{s}^{\sigma}\mathcal{B})$ of $F_t\odot F_s.$
	Note that $F_t\odot F_s=\overline{span}(\mathcal{B}\eta_{t}^{\sigma}\mathcal{B}\odot\mathcal{B}\eta_{s}^{\sigma}\mathcal{B})=\overline{span}(\mathcal{B}\eta_{t}^{\sigma}\odot\mathcal{B}\eta_{s}^{\sigma}\mathcal{B})=\overline{span}(\mathcal{B}\eta_{t}^{\sigma}\mathcal{B}\odot\eta_{s}^{\sigma}\mathcal{B}).$ For $s,t\in \mathbb{T},$ let $\{\mathcal{K}^{\sigma}_t\}_{\sigma\in S}$ be a $\mathfrak{K}_t$-family and $\{\mathcal{K}^{\sigma}_s\}_{\sigma\in S}$ be a $\mathfrak{K}_s$-family. Then by Theorem \ref{1234}, there exist isometries $
	\nu_t:{E}\bigodot F_t\to{E}$ and $\nu_s:{E}\bigodot F_s\to{E}$ with
	\begin{align*}
		\nu_t(x\odot \eta_t^{\sigma})=\mathcal{K}^{\sigma}_t(x)~~ and ~~ \nu_s(y\odot \eta_s^{\sigma})=\mathcal{K}^{\sigma}_s(y)~\mbox{for all}~y,x\in E,~\sigma\in S,s,t\in \mathbb{T}.
	\end{align*}
	It follows that
	$\mathcal{K}^{\sigma}_{s+t}(x)=\mathcal{K}^{\sigma}_s(\mathcal{K}^{\sigma}_t(x))=\nu_s(\mathcal{K}^{\sigma}_t(x)\odot \eta_s^{\sigma})=\nu_s(\nu_t\odot id_{F_s})(x\odot \eta_t^{\sigma}\odot \eta_s^{\sigma}).$ For $\sigma',\sigma\in S$ and $y,x\in E$  we get
	\begin{align*}\langle\mathcal{K}^{\sigma}_{s+t}(x),\mathcal{K}^{\sigma'}_{s+t}(y)\rangle&=\langle\mathcal{K}^{\sigma}_s(\mathcal{K}^{\sigma}_t(x)),\mathcal{K}^{\sigma'}_s(\mathcal{K}^{\sigma'}_t(y))\rangle=\mathfrak{K}_s^{\sigma,\sigma'}(\langle \mathcal{K}^{\sigma}_t(x),\mathcal{K}^{\sigma'}_t(y)\rangle)\\&=\mathfrak K_{s+t}^{\sigma,\sigma'}(\langle x,y\rangle)=\langle \eta^{\sigma}_{t+s}, \langle x,y\rangle\eta^{\sigma'}_{t+s}\rangle=\langle x\odot \eta^{\sigma}_{t+s}, y\odot\eta^{\sigma'}_{t+s}\rangle.
	\end{align*} 
	This computation provides us an isometry  $
	\nu_{s+t}:{E}\bigodot F_{s+t}\to{E}$ defined by
	$$\nu_{s+t}(x\odot \eta^{\sigma}_{t+s}):=\mathcal{K}^{\sigma}_{s+t}(x),$$ then the following diagram commutes:
	\[
	\begin{tikzcd}
		E\odot F_{s+t} \arrow{r}{I_E\odot \alpha_{t,s}} \arrow[swap]{dr}{\nu_{s+t}} & E\odot(F_t\odot F_s) \arrow{d}{\nu_s(\nu_t\odot I_{F_s})} \\
		& E
	\end{tikzcd}
	\]
	Hence we get $\nu_{s+t}=\nu_s(\nu_t\odot I_{F_s})(I_E\odot \alpha_{t,s})$ for $s,t\in \mathbb{T}.$ Let $\mathcal{E}_t=\overline{span}\{\mathcal{K}^{\sigma}_t(x)b: x\in E,b\in \mathcal{B},\sigma\in S, t\in \mathbb{T}\}\subseteq E,$ we obtained a unitary from $\nu_{s+t}$, by restricting its codomain to $\mathcal{E}_{s+t}$, we denote $\nu_{s+t}$ again. Define a $*$-homomorphism $\vartheta_{s+t}:B^{a}(E)\to B^{a}(\mathcal{E}_{s+t})$ by $\vartheta_{s+t}(a) = \nu_{s+t}(a \odot I_{F_{s+t}})\nu_{s+t}^*$ for $a\in B^{a}(E).$ Let $\mathcal{K}_{s+t}^{\sigma}(x)b\in \mathcal{E}_{s+t},$ then
	\begin{align*}
		\vartheta_{s+t}(a)\mathcal{K}_{s+t}^{\sigma}(x)b&=\nu_{s+t}(a \odot I_{F_{s+t}})\nu_{s+t}^*\mathcal{K}_{s+t}^{\sigma}(x)b\\&=\nu_s(\nu_t\odot I_{F_s})(I_E\odot \alpha_{t,s})(a \odot I_{F_{s+t}})(I_E\odot \alpha_{t,s}^*)(\nu_t^*\odot I_{F_s})\nu_s^*\mathcal{K}_{s+t}^{\sigma}(x)b\\&=\nu_s(\nu_t\odot I_{F_s})(I_E\odot \alpha_{t,s})(a \odot I_{F_{s+t}})(I_E\odot \alpha_{t,s}^*)(x\odot \eta_t^{\sigma}\odot \eta_s^{\sigma})b\\&=\nu_s(\nu_t\odot I_{F_s})(I_E\odot \alpha_{t,s})(a \odot I_{F_{s+t}})(x\odot \eta_{s+t}^{\sigma})b\\&=\nu_s(\nu_t\odot I_{F_s})(I_E\odot \alpha_{t,s})(ax\odot \eta_{s+t}^{\sigma})b=\nu_s(\nu_t\odot I_{F_s})(ax\odot \eta_t^{\sigma}\odot \eta_s^{\sigma})b\\&=\nu_s(\nu_t\odot I_{F_s})(a\odot I_{F_t}\odot I_{F_s})(\nu_t^*\odot I_{F_s})\nu_s^*\mathcal{K}_{s+t}^{\sigma}(x)b\\&=\nu_{s}(\vartheta_{t}(a)\odot I_{F_s})\nu_{s}^*\mathcal{K}_{s+t}^{\sigma}(x)b=\vartheta_{s}(\vartheta_{t}(a))\mathcal{K}_{s+t}^{\sigma}(x)b.
	\end{align*} Thus $\vartheta_{s+t}(a)=\vartheta_{s}(\vartheta_{t}(a))$ for all $a\in B^{a}(E).$ 
	%Identify $F_{s+t}$ with $B^a (\m B,F_{s+t})$ using $f\mapsto L_f$ where $L_f:c\mapsto fc$ and identify $\m B\bigodot F_{s+t}$ with $F_{s+t}$ using $b\odot f\mapsto bf$.  For each $x,y\in E$, $f,f'\in F_{s+t}$ and $b\in \m B$ we have
	%\begin{align*}
	%	\langle (x\odot id_{F_{s+t}})^* (y\odot f),b\odot f' \rangle &= \langle  y\odot f,x b\odot f' \rangle=\langle  f,\langle y,x b\rangle  f' \rangle= \langle  f,\langle y,x\rangle bf' \rangle\\ &=\langle  x^*yf, b f' \rangle=\langle  x^*y\odot f, b\odot  f' \rangle =\langle  (x^*\odot id_{F_{s+t}}) ( y \odot f), b\odot f' \rangle. 
	%\end{align*}
	%Therefore $(x\odot id_{F_{s+t}})^*=(x^*\odot id_{F_{s+t}})$ for all $x\in E.$ 
	It is easy to verify that 
	$$\begin{pmatrix}
		{\mathfrak{K}_s^{\sigma,\sigma'}} &  {\mathcal{K}_s^{\sigma^*}} \\
		\mathcal{K}_s^{\sigma'} & \vartheta_{s}
	\end{pmatrix}\begin{pmatrix}
		{\mathfrak{K}_t^{\sigma,\sigma'}} &  {\mathcal{K}_t^{\sigma^*}} \\
		\mathcal{K}_t^{\sigma'} & \vartheta_{t}
	\end{pmatrix}=\begin{pmatrix}
		\mathfrak{K}_{s+t}^{\sigma,\sigma'} &  {\mathcal{K}_{s+t}^{\sigma^*}} \\
		\mathcal{K}_{s+t}^{\sigma'}  & \vartheta_{s+t}
	\end{pmatrix} \quad \mbox{for all}\quad s,t\in \mathbb{T}.$$ 
	From Theorem \ref{DS5}, $\{\{\mathcal{K}^{\sigma}_t\}_{\sigma\in S}:t\in \mathbb{T}\}$ extends to block-wise semigroup of CPD-kernels $\begin{pmatrix}
		{\mathfrak{K}_t^{\sigma,\sigma'}} &  {\mathcal{K}_t^{\sigma^*}} \\
		\mathcal{K}_t^{\sigma'} & \vartheta_{t}
	\end{pmatrix}:\begin{pmatrix}
		\mathcal{B} &  E^* \\
		E & B^a(E)
	\end{pmatrix}\to \begin{pmatrix}
		\mathcal{B} &  \mathcal{E}_t^* \\
		\mathcal{E}_t & B^a(\mathcal{E}_t)
	\end{pmatrix}.$

	$(b)\Rightarrow (c):$ From Theorem \ref{DS5}, clearly the family of maps from $E_n$ to $E_n$ defined by
	\begin{equation}\label{PPP1}
		{\bf x}\mapsto({\mathcal{K}_t^{\sigma_1}}(x_1),{\mathcal{K}_t^{\sigma_2}}(x_2),\ldots,{\mathcal{K}_t^{\sigma_n}}(x_n))~\mbox{for}~{\bf x}=(x_1,x_2\ldots,x_n)\in E_n
	\end{equation}
	is a semigroup of completely bounded map. Since $\mathcal{K}_t^{\sigma}$ and $\mathcal{K}_s^{\sigma}$ are left $B^a({E})$-linear maps for $s,t\in \mathbb{T}.$ For $x\in E,a\in B^a(E)$ we have
	$${\mathcal{K}_{s+t}^{\sigma}}(ax)=\mathcal{K}_s^{\sigma}(\mathcal{K}_t^{\sigma}(ax))=\mathcal{K}_s^{\sigma}(\vartheta_{t}(a) \mathcal{K}_t^{\sigma}(x))=\vartheta_{s}(\vartheta_{t}(a))\mathcal{K}_s^{\sigma}(\mathcal{K}_t^{\sigma}(x))=\vartheta_{s+t}(a){\mathcal{K}_{s+t}^{\sigma}}(x).$$
	Hence $\mathcal{K}_{s+t}^{\sigma}$ is left $B^a ({E})$-linear function for every $\sigma\in S.$ Observe that $\mathcal{E}_{s+t}$ is a $B^a ({E})$-$\mathcal{B}$-correspondence with left module action is given by $\vartheta_{s+t}.$
	
	The rest of the proof $(d)\Leftrightarrow (c)$ and $(c)\Rightarrow (a)$ proceeds in the same as in \cite[Theorem 3.2]{DT17}.
\end{proof}

Now we recall the following: Let $\mathfrak{K}=(\mathfrak{K}_t)_{t\ge 0}$ be a semigroup of CPD-kernels over $S$ on a unital $C^*$-algebra $\mathcal{B},$ where $\mathcal{B}\subseteq B(\mathcal{H}).$ Let $(F_t,\eta_t)$ be the minimal Kolmogorov-representation for $\mathfrak{K}_t.$ Let ${K}_t=F_t\odot \mathcal{H}$ and denote by $\rho_t:b \mapsto b \odot I_{\mathcal{H}}$ be the Stinespring representation of $\mathcal{B}$ on ${K_t}.$ For $\sigma \in S,$ define a bounded map $L_{\eta_t^{\sigma}}:\mathcal{H}\to K_t$ defined by $L_{\eta_t^{\sigma}}(h)=\eta_t^{\sigma} \odot h$ for all $h\in \mathcal{H},t\in\mathbb{T}$ with $L_{\eta_t^{\sigma}}^*:\eta_t^{\sigma'}\odot h \mapsto \langle \eta_t^{\sigma},\eta_t^{\sigma'}\rangle h$ for $\sigma',\sigma\in S$ and $h\in \mathcal{H}$.
For $\sigma',\sigma\in S,h\in \mathcal{H}$ and $b\in \mathcal{B}$ we have  $L_{\eta_t^{\sigma}}^*\rho(b)L_{\eta_t^{\sigma'}}=\mathfrak{K}_t^{\sigma,\sigma'}(b).$ It is easy to verify that $\rho_t(\mathcal{B})L_{\eta_t^{\sigma}}\mathcal{H}=(\mathcal{B}\eta_t^{\sigma})\odot \mathcal{H}$ is total in $K_t.$ Hence, for all $t\in \mathbb{T},$ $(K_t,\rho_t,L_{\eta_t^{\sigma}})$ is the minimal Kolmogorov Stinespring representation for $\mathfrak{K}_t.$ The following result is a generalization of \cite[Corollary; Page No. 3]{Sk12}.

\begin{corollary}
	Let $\mathfrak{K}=(\mathfrak{K}_t)_{t\ge 0}$ be a semigroup of CPD-kernels over a set $S$ on a unital $C^*$-algebra $\mathcal{B},$ where $\mathcal{B}\subseteq B(\mathcal{H}).$ Let $(K_t,\rho_t,L_{\eta_t^{\sigma}})$ be the minimal Kolmogorov Stinespring representation for $\mathfrak{K}_t.$ Suppose that $\mathcal{K}^{\sigma}_t$ is a linear map on a Hilbert $C^*$-module $E$ over $\mathcal{B}$ for every $t\in \mathbb{T}$ and $\sigma\in S$ such that $\mathcal
	{K}=\{\{\mathcal
	{K}^{\sigma}_t\}_{\sigma\in S}:t\in \mathbb{T}\}$ is a $\mathfrak{K}$-family. Put $G_t:=E\odot K_t,$ and let $\xi_t$ be a representation of $E$ into $B(K_t,G_t)$ (induced by Stinespring representation $\rho_t$) defined as $\xi_t(x)=L_x,$ where $L_x(h')=x\odot h'$ for $x\in E,h'\in K_t.$ Put $L:=E\odot \mathcal{H},$ and define the representation $\chi$ of $E$ into $B(\mathcal{H},L)$ induced by the canonical injection $\mathcal{B}$ to $B(\mathcal{H}).$ Then by Theorem \ref{1234}, there exists an isometry $V_t:=v_t\odot I_{\mathcal{H}}\in B(E\odot F_t\odot \mathcal{H},E\odot \mathcal{H})=B(E\odot K_t,L)=B(G_t,L)$ such that $V_t\xi_t(x)L_{\eta_t^{\sigma}}=\chi\circ \mathcal{K}_t^{\sigma}(x)$ for $\sigma\in S,x\in E, t\in \mathbb{T}.$ Moreover 
	\begin{itemize}
		\item[(1)]  $(K_t,\rho_t,L_{\eta_t^{\sigma}})$ is uniquely determined by the properties which fulfil a minimal Stinespring construction.
		\item[(2)] $(G_t,\xi_t)$ is uniquely determined by the properties a representation induced by $\rho_t$ fulfills.
		\item[(3)] An isometry $V_t$ is determined uniquely by $V_t\xi_t(x)L_{\eta_t^{\sigma}}=\chi\circ \mathcal{K}_t^{\sigma}(x)$ for every $t\in \mathbb{T}.$ Also, $$V_{s+t}\xi_{s+t}(x)L_{\eta_{s+t}^{\sigma}}=\chi\circ \mathcal{K}_{s+t}^{\sigma}(x) \quad \mbox{for every}\quad t,s\in \mathbb{T}.$$
	\end{itemize}
\end{corollary}
\begin{proof}
	Let $x\in E,\sigma\in S,h\in \mathcal{H},t\in \mathbb{T}$ and from Theorem \ref{1234} we have
	\begin{align*}
		V_t\xi_t(x)L_{\eta_t^{\sigma}}h&=V_t\xi_t(x)(\eta_t^{\sigma}\odot h)=(v_t\odot I_{\mathcal{H}})(x\odot \eta_t^{\sigma}\odot h)=\mathcal{K}_t^{\sigma}(x)\odot h=\chi(\mathcal{K}_t^{\sigma}(x))h=(\chi\circ \mathcal{K}_t^{\sigma}(x))h.
	\end{align*}
	For $t,s\in \mathbb{T},$ let $\{\mathcal{K}^{\sigma}_t\}_{\sigma\in S}$ be a $\mathfrak{K}_t$-family and $\{\mathcal{K}^{\sigma}_s\}_{\sigma\in S}$ be a $\mathfrak{K}_s$-family. Then by Theorem \ref{1234}, there exist isometries $
	\nu_t:{E}\bigodot F_t\to{E}$ and $\nu_s:{E}\bigodot F_s\to{E}$ with
	\begin{align*}
		\nu_t(x\odot \eta_t^{\sigma})=\mathcal{K}^{\sigma}_t(x)~~ and ~~ \nu_s(y\odot \eta_s^{\sigma})=\mathcal{K}^{\sigma}_s(y)~\mbox{for}~~\sigma\in S,y,x\in E,s,t\in \mathbb{T}.
	\end{align*}
	From the proof of Theorem \ref{main result}, there exists an isometry  $
	\nu_{s+t}:{E}\bigodot F_{s+t}\to{E}$ such that $\nu_{s+t}=\nu_s(\nu_t\odot I_{F_s})(I_E\odot \alpha_{t,s}).$ For $x\in E,\sigma\in S,h\in \mathcal{H}$ we have
	\begin{align*}
		V_{s+t}\xi_{s+t}(x)L_{\eta_{s+t}^{\sigma}}h&=V_{s+t}\xi_{s+t}(x)(\eta_{s+t}^{\sigma}\odot h)=(v_{s+t}\odot I_{\mathcal{H}})(x\odot \eta_{s+t}^{\sigma}\odot h)\\&=(\nu_s(\nu_t\odot I_{F_s})(I_E\odot \alpha_{t,s})\odot I_{\mathcal{H}})(x\odot \eta_{s+t}^{\sigma}\odot h)=\nu_s(\nu_t\odot I_{F_s})(x\odot \eta_{t}^{\sigma}\odot \eta_{s}^{\sigma} )\odot h \\&=\nu_s(\mathcal{K}_t^{\sigma}(x)\odot \eta_{s}^{\sigma})\odot h=\mathcal{K}_{s+t}^{\sigma}(x)\odot h=
		\chi(\mathcal{K}_{s+t}^{\sigma}(x))h=(\chi\circ \mathcal{K}_{s+t}^{\sigma}(x))h.
	\end{align*}
\end{proof}

Now we recall the following setup and definitions from \cite{DT17}: Suppose that $\mathcal{K}=\{\{\mathcal{K}^{\sigma}_t\}_{\sigma\in S}:t\in \mathbb{T}\}$ is a $\mathfrak{K}$-family, where $\mathfrak K=\{\mathfrak K_t\}_{t\in \mathbb{T}}$ is a CPD-semigroup on $\mathcal{B}$ over a set $S.$ Then by Theorem \ref{1234}, there exists an isometry $\nu_t:{E}\bigodot F_t\to{E}$  such that
\begin{align*}
	\nu_t(x\odot \eta_t^{\sigma})=\mathcal{K}^{\sigma}_t(x)~\mbox{for}~~\sigma\in S,x\in E,t\in \mathbb{T},
\end{align*} where $(F_t, \eta_t)$ is the (minimal)
Kolmogorov-representation for $\mathfrak{K}_t.$ Further, assume that $\mathcal{E}_t$ is complemented in $E,$ then there exists $*$-homomorphism $\vartheta_{t}:B^{a}(E)\to B^{a}(E)$ given by $\vartheta_{t}(b) = \nu_{t}(b \odot I_{F_{t}})\nu_{t}^*$ for $b\in B^{a}(E)$ and the following diagram
\begin{eqnarray}
	\xymatrix{ \mathcal{B} \ar@{->}[r]^{ \mathfrak{K}^{\sigma,\sigma'}_t}
		\ar@{<-}[d]_{\langle y,\bullet y'\rangle}
		&\mathcal{B}
		\ar@{<-}[d]^{\langle \mathcal{K}^{\sigma}_t (y),\bullet \mathcal{K}^{\sigma'}_t(y')\rangle}
		\\
		B^a (E) \ar@{->}[r]^{\vartheta_t}
		&  B^a (E)
	}
\end{eqnarray} commutes for every $y,y'\in E$ and $t\in \mathbb{T}.$

\begin{definition}
	Suppose that $\mathfrak K$ is a CPD-semigroup over $S$ on a $C^*$-algebra $\mathcal{B}$ and ${E}$ is a Hilbert $C^*$-module over $\mathcal{B}.$ A semigroup of $*$-homomorphisms $\{\vartheta_t\}_{t\in \mathbb{T}}$ on $B^a (E)$ is called {\rm CPDH-dilation} of $\mathfrak K$ if there exists CPDH-semigroup $\{\{\mathcal
	{K}^{\sigma}_t\}_{\sigma\in S}:t\in \mathbb{T}\}$ on $E$ such that the following diagram 
	
	\begin{eqnarray}\label{diag5}
		\xymatrix{ \mathcal{B} \ar@{->}[r]^{ \mathfrak{K}^{\sigma,\sigma'}_t}
			\ar@{<-}[d]_{\langle y,\bullet y'\rangle}
			&\mathcal{B}
			\ar@{<-}[d]^{\langle \mathcal{K}^{\sigma}_t (y),\bullet \mathcal{K}^{\sigma'}_t(y')\rangle}
			\\
			B^a (E) \ar@{->}[r]^{\vartheta_t}
			&  B^a (E)
		}
	\end{eqnarray}
	commutes for every $y,y'\in E$ and $E$ is full.
\end{definition}

We will use the concept of the product systems to construct a CPDH-dilation of a CPD-semigroup. Suppose that $E$ is a full Hilbert $\mathcal{B}$-module and $(E^{\odot}=(E_t), u=(u_{s,t}))$ is a product system (resp. inclusion system). A family of unitaries (resp. isometries) $U_t:E\bigodot E_t\to E$ is called {\it a left dilation} of $E^{\odot}=(E_t)_{t\in \mathbb{T}}$ 
to $E$ if $$U_{s+t}=U_t(U_s\odot I_{E_t})(I_E\odot u_{s,t})\quad \mbox{for every}\quad s,t\in \mathbb{T}.$$
Suppose that $\mathfrak{K}$ is a CPD-semigroup over $S$ on a unital $C^*$-algebra $\mathcal{B}.$ From Section $4$ (see \cite[Section 4]{BBLS04}), there exist a product system $(\mathcal{E}^{\odot}=(\mathcal{E}_t)_{t\in \mathbb{T}},B=(B_{s,t})_{s,t \in \mathbb{T}})$ and a unit $\xi^{\sigma\odot}\in\mathcal{E}^{\odot}$ such that $\langle
\xi^{\sigma}_t,b \xi^{\sigma'}_t\rangle=\mathfrak{K}^{\sigma,\sigma'}_t(b)$
for every $t\in\mathbb{T},\sigma',\sigma\in S,b\in \mathcal{B}.$ Suppose that $\mathcal{E}$ is the inductive limit over $\mathcal{E}_t$. Fix $t\in \mathbb{T},$ using the bilinear unitaries $B_{s,t}$ we obtain the family of unitary $U_t:\mathcal{E}\bigodot \mathcal{E}_t\to \mathcal{E}$ ( that is, $U_t:=B_{s,t}^*$ as $s\to\infty$  ). Then the unitaries $U_t$
form a left dilation of $\mathcal{E}^{\odot}$ to
$\mathcal{E}$ (see \cite{Sk07,DT17}), and hence  $\vartheta_t$ on $B^a
(\mathcal{E})$ given by $\vartheta_t (b)=U_t (b\odot I_{\mathcal{E}_t})U_t^*$ for every
$t\in \mathbb{T},b\in B^a(\mathcal{E}),$ is an $E_0$-semigroup. Put $\mathcal{K}^\sigma_t
(y):=U_t (y\odot\xi^\sigma_t) $ for $t\in \mathbb{T},y\in E,\sigma\in S$, then the
diagram \ref{diag5} commutes.

\begin{remark}
	Consider a full Hilbert $C^*$-module $E$ over a given $C^*$-algebra $\mathcal{B}$ and $\mathcal{K}^{\sigma}_s:E\to E$  is a map 
	for all $s\in \mathbb{T}$ and $\sigma\in S.$ Suppose that $\mathcal{K}=\{\{\mathcal{K}^{\sigma}_t\}_{\sigma\in S}:t\in \mathbb{T}\}$ is a $\mathfrak{K}$-family, where $\mathfrak{K}=(\mathfrak{K}_t)_{t\in \mathbb{T}}$ is a semigroup of CPD-kernels over a set $S$ on $\mathcal{B}.$ Then by Theorem \ref{main result}, the family of isometries $\{\nu_t\}_{t\in \mathbb{T}}$
	(defined in Theorem \ref{main result}) form a left dilation of the inclusion system $F^{\odot} = (F_t)_{t\in \mathbb{T}}$ to $E.$
\end{remark}
	
	\subsection*{Acknowledgement}
	Dimple Saini is supported by UGC fellowship (File No:16-6(DEC. 2018)/2019(NET/CS\\IR)). Harsh Trivedi is supported by MATRICS-SERB Research Grant, File No: MTR/2021/000286, by the Science and Engineering Research Board (SERB), Department of Science \& Technology (DST), Government of India. Dimple Saini and Harsh Trivedi acknowledge Center for Mathematical \& Financial Computing (C-MFC) and the DST-FIST program (Govt. of India) for providing the financial support for setting up the computing lab facility under the scheme “Fund for Improvement of Science and Technology” (FIST - No. SR/FST/MS-I/2018/24) at the LNM Institute of Information Technology.

\end{document}